\documentclass[11pt,oneside,a4paper,mathscr]{amsart}

\usepackage{amsthm}
\usepackage{amsmath}
\usepackage{amssymb}
\usepackage{amsfonts}
\usepackage{amscd}
\usepackage{mathrsfs} 
\usepackage{mathtools}

\usepackage{bm}   
\usepackage[all]{xy}
\usepackage{ulem}
\makeatletter

\@addtoreset{equation}{section}
\makeatother   

\newtheorem{theorem}{Theorem}[section]
\newtheorem{lemma}[theorem]{Lemma}
\newtheorem{proposition}[theorem]{Proposition}
\newtheorem{corollary}[theorem]{Corollary}

\theoremstyle{definition}
\newtheorem{definition}[theorem]{Definition}
\newtheorem{example}[theorem]{Example}
\newtheorem{remark}[theorem]{Remark} 

\def\Z{\mathbb{Z}}

\def\C{\mathbb{C}}  

\def\id{\mathrm{id}}

\newcommand{\End}{\mathrm{End}}
\newcommand{\Hom}{\mathrm{Hom}}

\def\({\left(}
\def\){\right)}

\def\vac{|0\rangle}
\def\no{{\raise0.25em\hbox{$\mathop{\hphantom{\cdot}}
\limits^{_{\circ}}_{^{\circ}}$}}}

\newcommand{\nop}[1]{{\no} #1 {\no}}
\newcommand{\coh}[1]{| #1 \rangle}
\newcommand{\Coh}{\coh{{\rm coh}}}

\newcommand{\jump}[1]{\ensuremath{[\![#1]\!]} }
\newcommand{\pole}[1]{\ensuremath{(\!(#1)\!)} }

\newcommand{\Spec}{{\mathrm{Spec}\ \! }}

\newcommand{\ssc}{{Y}}

\newcommand{\Der}{\mathrm{Der}}
\def\ge{\mathfrak{g}}

\def\irr{{\mathfrak{f}}}

\def\Vir{\mathrm{Vir}}
\def\csm{\mathcal{M}}
\def\env{\mathcal{U}}
\def\lam{\lambda}
\def\O{{\mathcal{O}}}
\def\Heis{{\mathcal{F}}}
\def\HeisLie{\mathrm{Heis}}

\begin{document}

\title{Irregular vertex algebras}
\author{Akishi Ikeda, Yota Shamoto}

\address{Department of Mathematics, Graduate School of Science, Osaka University, 
Toyonaka, Osaka 560-0043, Japan}
\email{ikeda@math.sci.osaka-u.ac.jp}
\address{Kavli Institute for the Physics and Mathematics of the Universe (WPI),The University of Tokyo Institutes for Advanced Study, The University of Tokyo, Kashiwa, Chiba 277-8583, Japan}
\email{yota.shamoto@ipmu.ac.jp}

\maketitle

\begin{abstract}
We introduce the notion of irregular vertex (operator) algebras.
The irregular versions of fundamental properties, such as 
Goddard uniqueness theorem, associativity, 
and operator product expansions 
are formulated and proved.
We also give some elementary examples of irregular vertex operator algebras.
\end{abstract}
\setcounter{tocdepth}{1}
        \tableofcontents

\section{Introduction}
The vertex algebras, 
the definition of which was
introduced by Borcherds \cite{B} and
the foundation of the theory of which
was developed by Frenkel-Lepowsky-Meurman \cite{FLM}, 
may be seen as a
mathematical language of the  two-dimensional conformal field theory
initiated by Belavin-Polyakov-Zamolodchikov \cite{BPZ}.

Recently, 
several people  \cite{G, GT, JNS, NS, N1,N2} study 
irregular singularities
in conformal field theory.
They are
mainly motivated by Alday-Gaiotto-Tachikawa (AGT) correspondence \cite{AGT}
and their applications.
We would like to note that 
the notion of 
coherent states plays a fundamental role
in these studies.  

In the present paper, 
we shall initiate an attempt to give a mathematical language
of irregular singularities in conformal field theory
by introducing the notions of 
\textit{coherent state modules} and 
\textit{irregular vertex $($operator$)$ algebras}.
The main result of this paper is to formulate and prove
the irregular versions of fundamental properties
of irregular vertex algebras. 
We also give some elementary examples of
coherent state modules and irregular vertex algebras. 

In this introduction, we shall explain these notions and examples. 
We will explain the notion of coherent state module in Section \ref{intro CS}, 
irregular vertex algebras and their fundamental properties in Section \ref{intro IVA}, 
and the examples of irregular vertex algebras in Section \ref{intro example}. 
\subsection{Coherent states and irregular singularities in conformal field theory}\label{intro CS}
In conformal field theory on a Riemann sphere, 
Belavin-Polyakov-Zamolodchikov \cite{BPZ} and 
Knizhnik-Zamolodchikov \cite{KZ} found that 
the chiral correlation functions 
(conformal blocks) of vertex operators corresponding to 
highest weight vectors in minimal models and in Wess-Zumino-Witten models respectively, 
satisfy certain systems of differential equations with regular singularities. 
These equations are known as 
the BPZ equation and the KZ equation respectively.

Mathematically, 
these vertex operators can be interpreted as 
intertwining operators \cite{FHL} associated with 
highest weight vectors of modules over 
vertex algebras.  
The conformal block is then
given by the composition of intertwining operators \cite{H2}.

In \cite{AGT}, Alday-Gaiotto-Tachikawa found the relationship between Virasoro conformal blocks and 
Nekrasov partition functions of $\mathcal{N}=2$ superconformal gauge theories. 
The relation is now known as the AGT correspondence. 
In order to  generalize the AGT correspondence to Nekrasov partition functions of 
asymptotically free gauge theories, Gaiotto \cite{G} 
introduced irregular conformal blocks as the counter parts in CFT side. 

In his construction of the irregular conformal block, 
a coherent state, which is a  simultaneous eigenvector of 
some positive modes of the Virasoro algebra, plays a central role. 
He found that the coherent state corresponds to 
a state creating an irregular singularity of the stress-energy tensor (irregular state), 
while a highest weight vector corresponds to 
a state creating  a regular singularity (regular state). 

Correlation functions of vertex operators corresponding to irregular states 
are called irregular conformal blocks. 
More general studies of such irregular states for the Virasoro algebra 
were given in \cite{BMT, GT}. 
In particular, Gaiotto-Teschner \cite{GT} constructed an irregular state 
as a certain collision limit (confluence) of regular states and  
characterized this state  as an element of a $\mathcal{D}$-module. 
Irregular states for the affine Lie algebras and the $W_3$-algebra were studied in 
\cite{GL} and \cite{KMST}.

We note that in the side of mathematics, the idea that non-highest weight states of the affine Lie algebra 
create irregular singularities 
of the KZ equation was already appeared in \cite{FFT} and \cite{JNS} before \cite{G}. 

In Section \ref{CSM}, based on these various studies of 
coherent states and irregular singularities in conformal field theory, 
we introduce the notion of a coherent state module over a vertex algebra  (Definition \ref{csm}). 
Let $V$ be a vertex algebra and $S$ be a positively graded vector space over $\C$. 
Denote by $\mathcal{D}_S$ the ring of differential operators on $S$. 
The coherent state $V$-module $\csm$ on $S$ 
is a $\mathcal{D}_S$-module with $\mathcal{D}_S$-linear vertex operators
\begin{align*}
\ssc_\csm \colon V\longrightarrow \End_{\mathcal{D}_S}(\csm)\jump{z^{\pm 1}}, \quad 
A\mapsto Y_\csm(A, z)=\sum_{n\in \Z} A_{(n)}^\csm z^{-n-1}
\end{align*}
and a distinguished vector $\Coh \in \csm$, called a coherent state, together  
with some axioms. We call $S$ the space of internal parameters. 
The definition is motivated by the description of 
coherent states of the Virasoro algebra in \cite{GT} as follows:  
in the process of the confluence of $(r+1)$ regular states, the resulting irregular state obtains 
$r$ new parameters and the action of positive modes of the Virasoro algebra is given by 
differential operators of these $r$ parameters. 
We call them internal parameters of the coherent state. 

Conformal structures, the action of stress-energy tensors,  
on vertex algebras and their modules give coordinate change rules of vertex operators. 
They play important roles 
in coordinate-free approach for various concepts and theory on higher genus Riemann surfaces (see \cite{FBZ}). 
In Section \ref{ccsm}, we give the definition of a conformal structure on a coherent state module. 
The appearance of internal parameters makes the definition 
a little more complicated than the usual modules since we also need to consider coordinate changes of 
internal parameters.

In the subsequent, we will study irregular conformal blocks as the dual space of coinvariants
associated to conformal coherent state modules. 
As an application,  
the confluent KZ equation \cite{JNS} is described as an integrable connection on the irregular conformal 
block associated with coherent state modules over 
the vertex algebra $V_k(\mathfrak{sl}_2)$. 
In the future work, we will also discuss the relationship between $3$ points irregular conformal blocks of 
coherent state modules and the irregular type intertwining operators 
(see also Section \ref{intro IVA}).  

\subsection{Irregular vertex algebras}\label{intro IVA} 
The definition of a vertex algebras was introduced by 
Borcherds \cite{B} and   
a vertex operator algebra (vertex algebra with a conformal structure) 
was defined by Frenkel-Lepowsky-Meurman \cite {FLM}.
Nowadays, various equivalent definitions of vertex algebras are known. 
We shall adopt the axioms in \cite{FKRW} (see also \cite{FBZ, Kac}) known 
as Goddard's axioms \cite{God} since 
our definition of irregular vertex algebra is a generalization 
of them. 

The main point is the definition of vertex operators. 
For  a vertex algebra $V$ (see Definition \ref{def:VA}), 
vertex operators are given by  
the state-field correspondence
\[
\ssc \colon V\longrightarrow \End(V)\jump{z^{\pm 1}}, \quad 
A\mapsto Y(A, z)=\sum_{n\in \Z} A_{(n)} z^{-n-1}
\]
and they become fields, namely $Y(A, z)B  \in V((z))$ for all $A, B \in V$. 
They also satisfy the locality axiom
\[
(z-w)^N [Y(A,z), Y(B, w)] =0
\]
for sufficiently large $N$. As a consequence of the axioms, we have the operator product expansion (OPE) 
\[
Y(A,z) Y(B,w)= \sum_{n=0}^N 
\frac{Y(C_n, w)}{(z-w)^{n+1}}  
+ \no Y(A, z)Y(B,w) \no 
\]
where $C_n \in V$ are some states and $\no Y(A, z)Y(B,w) \no $ is the normally ordered product, which is 
smooth along with $z=w$. 
Thus in usual vertex algebras, all singularities are poles and 
this is the reason why correlation functions of 
usual vertex operators only have regular singularities.  

Our irregular vertex algebras are constructed on a particular coherent state modules, 
called envelopes of vertex algebras. A coherent state module 
$\env$ on a space of internal parameter $S$ is called an envelope of a vertex algebra $V$ 
if it contains $V$ in the fiber $\env_0$ on the origin $0 \in S$ 
and satisfies some compatibility conditions (see Definition \ref{def:envelope}). 
To consider irregular vertex operators for irregular states in $\env$, 
we also need to consider the singular locus 
$H \subset S$ since the composition of 
two irregular vertex operators may have singularities not only on $z =w$ 
but also on some divisor $H \subset S$. 
Therefore we need to introduce a $\mathcal{D}_S$-module $\env^{\circ}$ 
satisfies $\env \subset \env^{\circ}\subset \env(*H)$ with some good properties 
where $ \env(*H)$ is a localization of $\env$ along $H$ 
(see Definition \ref{FSL}). 

We can generalize the notion of field to have 
exponential type essential singularities. 
Denote by $\env_{\mu}$ the fiber of $\env$ on $\mu \in S$.
An irregular field with an irregularity $\irr(z;\lambda,\mu)$ and an internal 
parameter $\lam \in S$ is a 
$\Hom(\,\env_{\mu}^{\circ}, \overline{\env}_{\lam+\mu}^{\circ}\,)$-valued 
formal power series
\[
\mathcal{A}_{\lam}(z)=\sum_{n \in \Z}\mathcal{A}_{\lam,n}z^{-n-1} \in
\Hom(\,\env_{\mu}^{\circ}, \overline{\env}_{\lam+\mu}^{\circ}\,) \jump{z^{\pm 1}}
\]
satisfies the condition 
\[
\mathcal{A}_\lambda(z)\mathcal{B}_\mu\in 
e^{\irr(z;\lambda, \mu)}\env^\circ_{\lambda+\mu}\pole{z}
\]
for $\mathcal{B}_\mu \in \env_{\mu}^\circ$ where 
$\irr(z;\lam,\mu) =\sum_{k=1}^{2r}c_l(\lam,\mu) z^{-k}$
is a polynomial of $z^{-1}$ with coefficients $c_k(\lam,\mu) \in \mathcal{O}_{S^2}$ and 
$\overline{\env}_{\lam+\mu}^{\circ}$ 
is a certain completion of $\env_{\lam+\mu}^{\circ}$.
Thus the irregular field 
$\mathcal{A}_\lambda(z)$ has an exponential type essential singularity at $z=0$, 
but after dividing the factor $e^{\irr(z;\lambda, \mu)}$ it becomes a usual field. 
We also note that an irregular field with an internal parameter $\lam$ shifts 
internal parameters of states by $\lam$. 
Since the product of 
$e^{\irr(z;\lambda, \mu)}$ and an element in $\env^\circ_{\lambda+\mu}\pole{z}$ 
has infinite sums in 
$\env^\circ_{\lambda+\mu}$, we consider the completion 
$\overline{\env}_{\lam+\mu}^{\circ}$ 
and regard $\mathcal{A}_{\lam}(z)$  
as an element in 
$\Hom(\,\env_{\mu}^{\circ}, \overline{\env}_{\lam+\mu}^{\circ}\,) \jump{z^{\pm 1}}$.
This condition for irregular fields are 
first considered in \cite{N1} as a part of characterization of 
irregular vertex operators for the Virasoro algebra. 
In the definition of irregular vertex algebras, we assume that 
irregular vertex operators given by the state-field correspondence
\[
\ssc \colon
\env_\lambda^\circ\longrightarrow
\Hom
\(\env^\circ_{\mu},\overline{\env}^\circ_{\lambda+\mu}\)\jump{z^{\pm 1}},
\]
become irregular fields with a fixed irregularity $\irr(z;\lambda, \mu)$.

We also require that these irregular fields satisfy the irregular locality axiom
\[
(z-w)^N\(e^{-\irr(z-w;\lambda,\mu)}_{|z|>|w|}Y(\mathcal{A}_\lambda,z)Y (\mathcal{B}_\mu, w)
-e^{-\irr(z-w;\lambda,\mu)}_{|w|>|z|}Y (\mathcal{B}_\mu, w)Y(\mathcal{A}_\lambda,z)\)=0
\]
for $\mathcal{A}_\lam \in \env_\lam^\circ$, $\mathcal{B}_\mu \in \env_\mu^\circ$
and sufficiently large $N$. 
Here, $e^{-\irr(z-w;\lambda,\mu)}_{|z|>|w|}$ and $e^{-\irr(z-w;\lambda,\mu)}_{|w|>|z|}$
denote the expansions of $e^{\irr(z-w;\lambda,\mu)}$
to their respective domains 
(see Section \ref{twist_Lie} for more detail).
As a consequence, 
we can show the following:
\begin{theorem}
 [{Theorem \ref{ope theorem}}. OPE for irregular vertex operators]
\[
Y(\mathcal{A}_{\lam}, z)Y(\mathcal{B}_{\mu},w) =
e^{\irr(z-w;\lam,\nu)}\left(\sum_{n=0}^N 
\frac{Y(\mathcal{C}_{\lam+\mu, n}, w)}{(z-w)^{n+1}}  
+ \no Y(\mathcal{A}_{\lam}, z)Y(\mathcal{B}_{\mu},w) \no \right)
\]
where $\mathcal{C}_{\lam+\mu, n}\in \env_{\lam+\mu}^{\circ}$ are some states 
with an internal parameter $\lam +\mu$. 
\end{theorem}
We need to note that the definition of normally ordered product 
for irregular fields 
is a littele complicated (see Definition \ref{NOP}). 

Another new feature of 
irregular vertex operators is that the state-field correspondence is 
a $\mathcal{D}$-module homomorphism, which implies that it satisfies
\begin{align}\label{DIFF}
[\partial_\lam, Y(\mathcal{A}_\lam, z)]= 
Y(\partial_\lam\mathcal{A}_\lam, z)
\end{align}
for a vector field $\partial_\lam$ along the direction of the parameter $\lambda$
(See Remark \ref{RMK} for more precise). 
This yields additional integrable differential equations on 
the space of internal parameters $S$ for correlation functions.

The irregular vertex operator given in the present paper 
is a prototype of a more general object, the irregular type intertwining operator 
among general conformal coherent state modules. 
(See \cite{FHL},\cite{H1, H2} for regular type intertwining operators.)
Actually, Nagoya \cite{N1} considered the irregular type intertwining operator 
among two coherent state modules and one highest weight module. 
In the future work, we will give the definition of the irregular 
type intertwining operator.

\subsection{Examples of irregular vertex algebras}\label{intro example}
We will give two classes of elementary examples 
of irregular vertex algebras in Section \ref{FFIVA} and Section \ref{ffr}.
We see basic ideas for the easiest case.

In Section \ref{FFIVA}, we define irregular version of Heisenberg vertex algebra $\Heis^{(r)}$. 
The construction is based on the ideas of \cite{NS}. In the following, we consider the case $r=1$. 
Recall the Heisenberg Lie algebra 
\[
\HeisLie=\bigoplus_{n \in \Z}\C a_n \oplus \C \mathbf{1}
\]
with the relations $[a_n, a_m]=  n\,\delta_{m+n, 0}\mathbf{1}$ and $[a_n, \mathbf{1}]=0$. Let $\vac$ be 
the vacuum characterized by $a_n \vac =0$ for $n \geq 0$ and consider the Fock space $\Heis$ 
generated by $\vac$ over $\HeisLie$. Then, the Fock space 
$\Heis = \bigoplus \C a_{-n_1}\cdots a_{-n_k}\vac$ has the vertex algebra structure by the 
state-field correspondence
\[
Y(a_{-n_1}a_{-n_2}\cdots a_{-n_k}\vac, z)
\coloneqq \nop{\partial_z^{(n_1-1)}a(z)\partial_z^{(n_2-1)}a(z)\cdots \partial_z^{(n_k-1)}a(z)}
\]
where $a(z)=\sum_{n \in \Z}a_n z^{-n-1}$ is the bosonic current and $\partial_z^{(n)}= (n !)^{-1}\partial_z^n$. 
Bosonic currents $a(z)$ and $a(w)$ have the OPE
\[
a(z)a(w)=\frac{\mathbf{1}}{(z-w)^2}+ \nop{a(z)a(w)}.
\]
Now we consider the coherent state $\coh{\lam}$ characterized by $a_1 \coh{\lam}=\lam \coh{\lam}$ and 
$a_n \coh{\lam}=0$ for $n=0$ and $n \geq 2$. 
We can realize this coherent state as $\coh{\lam}= e^{\lam a_{-1}} \vac$ in a certain completion of 
${\Heis}\otimes \C[\lambda]$ with $\deg \lam = -1$. 
We set $S_{\lam}=\Spec \C[\lam]$ and assume it the space of internal parameters. 
Consider a family of Fock spaces on $S_{\lam}$ 
generated by $\coh{\lam}$ over $\HeisLie$ and denote it by $\Heis_{\lam}^{(1)}$. 
Then $\Heis_{\lam}^{(1)}=\bigoplus \C[\lambda]a_{-n_1}\cdots a_{-n_k}\coh{\lambda}$. 
The space $\Heis_{\lam}^{(1)}$ is called the envelope of the Heisenberg vertex algebra $\Heis$. 
Let us consider the vertex operators corresponding to states in $\Heis_{\lam}^{(1)}$. For the 
coherent state $\coh{\lam}$, 
since $\coh{\lam}=\sum_{k=0}^{\infty}  \frac{a_{-1}^k \vac}{k!}$, it is natural to assume
\[
Y(\coh{\lam},z )=\sum_{k=0}^{\infty}\frac{Y(a_{-1}^k \vac, z)}{k!}
=\sum_{k=0}^{\infty}\frac{ \nop{a(z)^k} }{k!}= \nop{e^{\lam a(z)}}
\]
where $\nop{ e^{\lam a(z)} }= e^{\lam a_+( z)}e^{\lam a_-(z)} $. 
We can easily check that $Y(\coh{\lam}, z) \vac|_{z=0}= \coh{\lam}$.

The vertex operator $Y(\coh{\lam}, z)$ is not a field in the usual sense but an  
irregular filed with the irregularity $\irr(z;\lambda, \mu)=\lambda\mu/z^2$ since the direct computation by 
using the Baker-Campbell-Hausdorff formula gives
\[
Y(\coh{\lam},z )\coh{\mu} = \nop{e^{\lam a(z)}} e^{\mu a_{-1}}\vac =e^{\lam \mu \slash z^2} (\coh{\lam+\mu}+o(z))
\]
where $o(z)$ is positive powers of $z$.
So we call $Y(\coh{\lam}, z)$ the irregular vertex operator. 
For a general state $a_{-n_1}\cdots a_{-n_k}\coh{\lambda}$, we can define irregular vertex operators as
\[
Y(a_{-n_1}\cdots a_{-n_k}\coh{\lambda}, z) = \nop{\partial_z^{(n_1-1)}a(z)\cdots \partial_z^{(n_k-1)}a(z) e^{\lam a(z)}}.
\]
The property that the vertex operation $Y$ is a $\mathcal{D}$-module homomorphism implies
\[
[\partial_{\lam}, Y(\coh{\lam}, z)]= Y(\partial_{\lam} \coh{\lam}, z)= Y(a_{-1}\coh{\lam},z)= \nop{a(z) Y(\coh{\lam}, z)}. 
\]
Actually, $Y(\coh{\lam},z) =\nop{e^{\lam a(z)}}$ is the solution of 
the differential equation $[\partial_{\lam}, Y(\coh{\lam}, z)]=\nop{a(z) Y(\coh{\lam}, z)}$ under the 
initial condition $Y(\vac, z)= \id$.

Thus, we can define irregular vertex operators 
$Y\colon {\Heis}^{(1)}_\lambda\to 
\Hom({\Heis}^{(1)}_\mu, \overline{{\Heis}}_{\lambda+\mu}^{(1)})\jump{z^{\pm 1}}$. 
Finally, we see an example of the OPE for irregular vertex operators. 
For $Y(\coh{\lam}, z)$ and $Y(a_{-1}\coh{\mu}, z)$, we have
\begin{align*}
Y(\coh{\lam}, z)Y(a_{-1}\coh{\mu}, w) = e^{\lambda\mu/(z-w)^2}_{|z|>|w|} 
\left(\frac{\lam Y(\coh{\lam+\mu}, w)}{(z-w)^2} \right. &+ \frac{\lam Y(a_{-2}\coh{\lam+\mu}, w) }{z-w} \\
&+\nop{ Y(\coh{\lam}, z)    Y(a_{-1}\coh{\mu}, w)   } ).
\end{align*}
It follows from the computation 
\[
Y(\coh{\lam}, z) a_{-1} \coh{\mu} =  e^{\lam \mu \slash z^2}\left(\frac{\lam \coh{\lam+\mu}}{z^2}+ 
\frac{\lam a_{-2} \coh {\lam + \mu}}{z} +o(1)  \right).
\]

In Section \ref{ffr},
we will define the irregular version $\Vir_c^{(r)}$ of the Virasoro vertex algebra $\Vir_c$
for $r\in\Z_{>0}$. 
The space of internal parameter of $\Vir_c^{(r)}$ will be the same as that of ${\Heis}^{(r)}$.
The construction will be given 
via the free field realization of Virasoro vertex algebra to
the Heisenberg vertex algebra. 
The main difference is that $\Vir_c^{(r)}$ have singularity in the space of internal parameters. 
In the case $r=1$ the singular locus will be $\lambda=0$. 
This kind of singularity also appears in the definition of irregular vertex operators 
for Virasoro Verma modules in \cite{N1}.

\subsection{Notations}\label{notation}
Throughout this paper, the term ``grading'' refers to the $\Z$-grading. 
For a graded module $M=\bigoplus_{k\in\Z}M_k$ over a graded $\C$-algebra $R=\bigoplus_{k \in \Z}R_k$,
the symbol $M\jump{z^{\pm 1}}=\bigoplus_\ell M\jump{z^{\pm 1}}_\ell$ denote  the graded $R$-module
whose degree $\ell$ part $M\jump{z^{\pm 1}}_\ell$ is 
\begin{align*}
M\jump{z^{\pm 1}}_\ell\coloneqq \prod_{n\in\Z} M_{\ell+n}z^n.
\end{align*}
In other words, we set $\deg z=-1$ and only consider 
finite sums of homogeneous power series in this paper.
We define modules  $M\jump{z}$ (resp. $M\pole{z}$) of positive formal power series (resp. formal Laurent series)
with coefficients in $M$ in a similar way.

For two graded $R$-modules $M, N$ and an integer $n$, 
$\Hom_R(M, N)_n$ denotes the $\C$-vector space of 
$R$-linear morphisms of degree $n$. 
We then put 
\begin{align*}
\Hom_R(M, N)=\bigoplus_{n\in\Z}\Hom_R(M, N)_n
\end{align*}
and regard it as a graded $R$-module.
We also set $\End_R(M)\coloneqq \Hom_R(M,M)$.

The derivation with respect to a (local, formal,...) coordinate $x$ is denoted by 
$\partial_x$ or $\frac{\partial}{\partial x}$. 
We also set $\partial_x^{(n)}\coloneqq (n!)^{-1}\partial_x^n$ for $n\in\Z_{\geq 0}$.

For an affine scheme $X=\Spec A$, we identify the structure sheaf $\mathcal{O}_X$ with 
the ring $A$ of global sections as an abuse of the notation.  
Sheaves of modules over $\mathcal{O}_X$ are also identified with 
the $A$-modules of their global sections.
In this paper, we only consider the case 
$X=\Spec \C[x_1,\dots,x_n]$ for some $n$.
Then we denote $\C[x_1,\dots, x_n]$ by $\mathcal{O}_X$. 
We also set $\Theta_X\coloneqq \bigoplus_{j=1}^n\mathcal{O}_X\partial_{x_j}$ 
and $\mathcal{D}_X=\C\langle x_i, \partial_{x_j}\mid i, j=1,\dots, n\rangle$ with 
the relation $[\partial_{x_i}, x_j]=\delta_{i, j}$.


\subsection*{Acknowledgement}
The first author would like to thank to Hajime Nagoya for teaching basic ideas of 
his various works on irregular conformal blocks. 
The second author would like to 
thank Takuro Mochizuki and Jeng-Daw Yu for their encouragement.  
This work was supported by World Premier International Research Center Initiative (WPI), MEXT, Japan.
The first author is supported by JSPS KAKENHI Grant
Number 16K17588 and 16H06337.
The second author is supported by JSPS KAKENHI Grant Number JP18H05829.

\section{Coherent state modules}\label{CSM}
In this section, we introduce the notion of 
coherent state modules.

\subsection{Space of internal parameters}\label{ip}
Let $\C[\lambda_j]_{j\in J}$ be the 
graded ring of polynomials with variables $\lambda_j$
indexed by a finite set $J$.
We assume that the degree 
 $d_j\coloneqq \deg \lambda_j$ are all negative.
We call the spectrum $S\coloneqq \Spec \C[\lambda_j]_{j\in J}$ 
a \textit{space of internal parameters}. 
We set $\mathcal{O}_S\coloneqq \C[\lambda_j]_{j\in J}$ to simplify the notation
(See Section \ref{notation}).
We consider $S$ as an additive algebraic group in a natural way.
The addition morphism is denoted by
\begin{align*}
\sigma \colon S\times S\longrightarrow S.
\end{align*}
The grading defines a $\C^*$-action on $S$ naturally. 
The addition $\sigma$ is $\C^*$-equivariant.

Let $\mathcal{D}_S=\C\langle \lambda_j,\partial_{\lambda_j}\rangle_{j\in J}$ 
denote the ring of differential operators on $\mathcal{O}_S$.
This is a graded algebra such that $\deg \partial_{\lambda_j}=-d_j$.
We also set 
$\Theta_S\coloneqq \bigoplus_{j\in J}\mathcal{O}_S\partial_{\lambda_j}\subset \mathcal{D}_S$, 
which is naturally equipped with the structure of a graded Lie algebra.  
Let $\mathrm{Der}(\C\jump{t})$ denote the Lie algebra 
$\C\jump{t}\partial_t$ with the usual Lie bracket, i.e.
\begin{align*}
&[t^{k+1}\partial_t, t^{\ell+1}\partial_t]=(\ell-k)t^{k+\ell+1}\partial_t
&(k, \ell\in\Z_{\geq 0}).
\end{align*}
The grading is given by $\deg t=-1$, $\deg\partial_t=1$.
Let $\Der_0(\C\jump{t})$ denote the Lie subalgebra $t\C\jump{t}\partial_t$ of $\Der(\C\jump{t})$.

\begin{definition}\label{conf S}
A \textit{$\Der_0(\C\jump{t})$-structure} on $S$ is a grade preserving Lie algebra homomorphism
\begin{align*}
\rho_S \colon {\Der}_0(\C\jump{t})\longrightarrow \Theta_S,\quad t^{k+1}\partial_t\mapsto D_k
\end{align*}
such that $[D_0, f]=\ell f, \text{ and}$
\begin{align*}
[\rho_{S^2}(t^{k+1}\partial_t), \sigma^*f]=\sigma^*[D_k, f]
\end{align*}
for $f\in \mathcal{O}_{S, -\ell}$, $k\in\Z_\geq0$, 
where $\rho_{S^2} \colon \Der_0(\C\jump{t})\to \Theta_{S^2}$
denotes the diagonal action induced from $\rho_{S}$.
\end{definition}

\begin{example}[see  {\cite[Lemma 4.5]{I}}]\label{S}
Fix $r\in\Z_{>0}$. 
Let  $S\coloneqq \Spec \C[\lambda_j]_{j=1}^r$ be a space of internal parameter with $\deg \lambda_j=-j$.
Set 
\[
D_k\coloneqq \sum_{j =1}^{r-k}j\lambda_{j+k}\frac{\partial}{\partial \lambda_j}\quad
(k=0,\dots, r-1)
\]
and $D_k=0$ for $k\geq r$.
The morphism 
$\rho_S \colon {\Der}_0(\C\jump{t})\longrightarrow \Theta_S$,
$t^{k+1}\partial_t\mapsto D_k$ defines 
the $\Der_0(\C\jump{t})$-structure 
on $S$.
\end{example}
\subsection{Coherent state modules}\label{S csm}
We firstly recall the definition of vertex algebras. 
In this paper, we only consider $\Z$-graded vertex algebras:
\begin{definition}\label{def:VA}
A (graded) \textit{vertex algebra} is 
a tuple $V=(V, \vac, T, Y)$ of a graded vector space $V=\bigoplus_{n\in\Z}V_n$, 
a non-zero vector $\vac\in V_0$, a degree one endomorphism $T$ on $V$, and 
a grade-preserving homomorphism 
\begin{align*}
Y \colon V\longrightarrow \End(V)\jump{z^{\pm 1}}, \quad
 A\mapsto Y(A, z)=\sum_{n\in \Z}A_{(n)}z^{-n-1}
\end{align*}
with the following properties:
\begin{itemize}
\item
(vacuum axiom) $Y(\vac, z)=\id_V$ and $Y(A, z)\vac\in A+zV\jump{z}$ for any $A\in V$.
\item
(translation axiom)
$[T, Y(A, z)]=\partial_zY(A, z)$ for any $A\in V$ and $T\vac=0$.
\item (field axiom) $Y(A, z) B\in V\pole{z}$ for any $A, B\in V$. 
\item
(locality axiom) For any $A, B\in V$, there exists a positive integer $N$ such that 
\begin{align*}
(z-w)^N[Y(A, z), Y(B, w)]=0 
\end{align*}
in $\End(V)\jump{z^{\pm 1},w^{\pm 1}}$.
Here, $w$ denotes a copy of $z$ with $\deg w=-1$.
\end{itemize}
\end{definition}

We then consider the following family of modules over $V$. 
\begin{definition}[Coherent state modules]\label{csm}
Let $V$ be a vertex algebra. Let $S$ be a space of internal parameters (Section \ref{ip}).
A \textit{coherent state $V$-module on $S$} is a triple $(\csm, Y_\csm, \Coh)$
of a $\C$-graded $\mathcal{D}_S$-module $\csm$, 
a degree zero
$\C$-linear map
\begin{align}\label{Y_U}
Y_\csm \colon V\longrightarrow \End_{\mathcal{D}_S}(\csm)\jump{z^{\pm 1}}, \quad 
A\mapsto Y_\csm(A, z)=\sum_{n\in \Z} A_{(n)}^\csm z^{-n-1},
\end{align}
and a homogeneous global section $\Coh$ of $\csm$ 
(called \textit{a coherent state} of $\csm$)
such that  the following properties hold:
\begin{enumerate}
\item $Y_\csm(\vac, z)=\id_\csm$.
\item $Y_\csm(A, z)\mathcal{B}$ is in $\csm\pole{z}$ for any $A\in V$ and $\mathcal{B}\in \csm$.
\item For any $A, B\in V$ and $\mathcal{C}\in \csm$,  
         the three elements
         \begin{align*}
         &Y_\csm(A, z)Y_\csm(B, w)\mathcal{C}\in \csm\pole{z}\pole{w}\\
         &Y_\csm(B, w)Y_\csm(A, z)\mathcal{C}\in \csm\pole{w}\pole{z}\\
         &Y_\csm(Y(A, z-w)B, w)\mathcal{C}\in \csm\pole{w}\pole{z-w}
         \end{align*}
         are the expansions of the same element in $\csm\jump{z, w}[z^{-1}, w^{-1}, (z-w)^{-1}]$
         to their respective domains.
\item $\csm$ is generated by $\Coh$ over $V$ and $\mathcal{D}_S$, i.e. if a $\mathcal{D}_S$-submodule 
         $\csm'\subset \csm$ contains $\Coh$ and is closed under operations $A_{(n)}^\csm$ for any $A\in V$ and $n\in\Z$, 
         then we have $\csm'=\csm$.
\item Let $\csm_\mathcal{O}$ denote the smallest graded $\mathcal{O}_S$-submodule of $\csm$ such that 
         $\Coh$ is contained in $\csm_\mathcal{O}$ and closed under operations $A_{(n)}^\csm$ 
         for any $A\in V$ and $n\in \Z$.
         Then, the support of the quotient module $\csm/\csm_\mathcal{O}$ is either an
         empty set or co-dimension one subvariety in $S$.
\end{enumerate}
\end{definition}
\begin{remark}\label{csmrem}
Let $\widetilde{U}(V)$
be the complete topological associative algebra associated to $V$ (see \cite[Definition 4.3.1]{FBZ}). 
By the same argument as \cite[Theorem 5.1.6]{FBZ}, the condition $(1)$, $(2)$, and $(3)$
is equivalent to the condition that $\csm$ is a smooth $\widetilde{U}(V)$-module.  
By definition, the action of $\widetilde{U}(V)$ 
 is compatible with the action of $\mathcal{D}_S$,
and hence we may consider $\csm$ as a module over
$\widetilde{U}(V)\otimes_\C\mathcal{D}_S$. 
The condition $(4)$ is then equivalent to the condition 
$\csm=(\widetilde{U}(V)\otimes_\C\mathcal{D}_S)\Coh$. 
The module $\csm_\mathcal{O}$ in the condition $(5)$ is expressed as 
$\csm_\mathcal{O}=(\widetilde{U}(V)\otimes_\C\mathcal{O}_S)\Coh$.
\end{remark}

\begin{definition}[Singular locus of coherent state modules]
Let $\csm$ be a coherent state $V$-module. Let $\csm_\mathcal{O}$ be the 
$\mathcal{O}_S$-submodule 
defined in (5) of Definition \ref{csm}.  
The \textit{singular locus $H$ of the coherent state module
$\csm$} is the support of $\csm/\csm_\mathcal{O}$:
\begin{align*}
H\coloneqq \mathrm{Supp}(\csm/\csm_\mathcal{O}),
\end{align*} 
which is by assumption a $\C^*$-invariant Zariski closed subvariety in $S$.
The coherent state module $\csm$ is called \textit{non-singular along $S$} if $H$ is empty.
\end{definition}

\subsection{Conformal coherent state modules}\label{ccsm}
Recall that the \textit{Virasoro algebra} is a  
Lie algebra $\Vir=(\bigoplus_{n\in \Z} \C L_n)\oplus \C C$ whose Lie bracket is defined 
as follows: 
\begin{align}\label{Virasoro}
&[L_m, L_n]=(m-n)L_{m+n}+\frac{1}{12}(m^3-m)\delta_{m+n,0}C,
&&[L_n, C]=0 &(m,n\in \Z).
\end{align}

Fix a complex number $c\in \C$.  
Let $\C_c$ be an one dimensional representation of the Lie subalgebra 
$\Vir_{\geqslant -1}\coloneqq \bigoplus_{n\geq -1}\C L_n\oplus \C C$
defined by $L_n1=0$ for $n\geq -1$ and $C\cdot1=c1$. 
Take a induced representation
\begin{align*}
\Vir_c\coloneqq  U(\Vir)\otimes_{U(\Vir_{\geqslant -1})}\C_c\simeq 
\bigoplus_{n_1\geq \cdots \geq n_k>1}\C L_{-n_1}\cdots L_{-n_k}v_c
\end{align*}
where $U(\Vir)$ (resp. $U(\Vir_{\geqslant -1})$) denotes 
the universal enveloping algebra of $\Vir$ 
(resp. $\Vir_{\geqslant -1}$), 
and $v_c$ denotes the image of $	1\otimes 1$.
Define the $\Z$-gradation on $\Vir_c$ by the formulas 
$\deg L_{-n}=n$ and $\deg v_c=0$.
The power series 
\begin{align*}
T(z)\coloneqq \sum_{n\in \Z}L_n z^{-n-2}
\end{align*}
defines a field on $\Vir_c$.

A \textit{Virasoro vertex algebra with central charge $c$} 
is a vertex algebra 
\begin{align*}
\Vir_c\coloneqq (\Vir_c, v_c, L_{-1}, Y(\cdot, z))
\end{align*} where 
$Y(\cdot, z):\Vir_c\to \End(\Vir_c)\jump{z^{\pm 1}}$ is defined as follows:
\begin{align*}
Y( L_{-n_1}\cdots L_{-n_k}v_c, z)\coloneqq \nop{\partial_z^{(n_1-2)}T(z)\cdots \partial_z^{(n_k-2)}T(z)}
\end{align*}
where $\nop{\cdot}$ denotes the normally ordered product.

A \textit{conformal structure of central charge $c\in \C$}
on a vertex algebra $V$ is 
a degree two non-zero vector $\omega\in V_2$ called 
a \textit{conformal vector}
such that 
the Fourier coefficients 
$L_n^V$ of the corresponding vertex operator 
\begin{align*}
Y(\omega, z)=\sum_{n\in\Z} L_n^V z^{-n-2}
\end{align*}
satisfy $L_{-1}^V=T$ and $L_2^V\omega=\frac{c}{2}\vac$.
A vertex algebra with a conformal structure is called a \textit{vertex operator algebra}, or
a \textit{conformal vertex algebra}.
The following properties of vertex operator algebra is well known (see \cite[Lemma 3.4.5]{FBZ}):
\begin{itemize}
\item
$\{L_n^V\}_n$ satisfies the relation (\ref{Virasoro}) replacing $C$ by $c\cdot \id_{V}$.
\item $L_0^V=n\cdot \id_{V_n}$ on $V_n$ for every $n\in \Z$. 
\item There exists a unique morphism $\Vir_c\to V$ of vertex algebras such that 
 $v_c\mapsto \vac$, $L_{-2}v_c\mapsto \omega$.
\end{itemize}
\begin{definition}
Let $\csm=(\csm, Y_\csm, \Coh)$ be a coherent state $V$-module over $S$. 
Assume that $V$ has a conformal vector $\omega$, 
and $S$ has a $\Der_0(\C\jump{t})$-structure 
\[
\rho_S \colon \Der_0\C\jump{t}\to \Theta_S, \quad t^{k+1}\partial_t\mapsto D_k
\]
(see Definition \ref{conf S}).
Let $L^\csm_n$ be the Fourier coefficients of the action
         \begin{align*}
         Y_\csm(\omega, z)=\sum_{n\in\Z} L_n^\csm z^{-n-2}.
         \end{align*}
Then, $\csm$ is called \textit{conformal} if 
there exist differential operators 
$\mathcal{L}_k=h_k+D_k\in \mathcal{O}_S\oplus \Theta_S$
such that the following properties hold:
\begin{enumerate}
\item $L_k^\csm\Coh=\mathcal{L}_k\Coh$ for  $k\in \Z_{\geq 0}$. 

\item The map 
\[
\rho_\csm \colon \Der_0\C\jump{t}\to \End_\C(\csm),\quad
 t^{k+1}\partial_t\mapsto-(L_k^\csm-\mathcal{L}_k)
\]
is a Lie algebra homomorphism.
\item $L_0^\csm-D_0$ is the grading operator on $\csm$ 
(i.e. $(L_0^\csm-D_0)m=\deg (m)m$ for any homogeneous section $m\in \csm$)
         and $L_k^\csm-\mathcal{L}_k$
         is locally finite for $k>0$
         (i.e. for any $m\in\csm$, there exists $N>0$ such that  $(L_k^\csm-\mathcal{L}_k)^Nm=0$).
\end{enumerate}
\end{definition}

\subsection{An example of conformal coherent state module}\label{Vir csm}
We shall give an example of conformal coherent state module, 
following the idea of Gaiotto-Teschner \cite[Section 2.1.2]{GT}. 
Take the space of internal parameter $S$ as in Example \ref{S} for fixed $r>0$.
Fix complex numbers $\rho$ and $\lambda_0$.
Set $h_k(\lam)\coloneqq \frac{1}{2}\sum_{j=0}^k\lambda_j\lambda_{k-j}-\rho(k+1)\lambda_k$
for $k=0,\dots, 2r$, where we put $\lambda_\ell=0$ for $\ell>r$. 
Put
$h\coloneqq  h_0=2^{-1}\lambda_0(\lambda_0-2\rho)$.
Put $\mathcal{L}_k\coloneqq  h_k+D_k$ for $k=0,\dots, 2r$ and $\mathcal{L}_m\coloneqq 0$ for $m> 2r$. 
\begin{lemma}\label{gyaku virasoro}
For $k,\ell\in \Z_{\geq 0}$, we have
\begin{align*}
[\mathcal{L}_k, \mathcal{L}_\ell]=(\ell-k)\mathcal{L}_{k+\ell}.
\end{align*}
\end{lemma}

Let $\Vir_{\geqslant 0}$ denote the Lie subalgebra $\bigoplus_{n\geq 0}\C L_n\oplus \C C$ of $\Vir$.
By Lemma \ref{gyaku virasoro}, 
\begin{align*}
L_n\cdot P(\lambda, \partial_\lambda)\coloneqq P(\lambda, \partial_\lambda)\mathcal{L}_n, 
&&C\cdot P(\lambda, \partial_\lambda)\coloneqq cP(\lambda, \partial_\lambda)
&&(n\geq 0, P(\lambda, \partial_\lambda)\in\mathcal{D}_S)
\end{align*}
defines left $U(\Vir_{\geqslant 0})$-module structure on $\mathcal{D}_S$, 
where $U(\Vir_{\geqslant 0})$ denotes the universal enveloping algebra of $\Vir_{\geqslant 0}$.

\begin{definition}
We set  
\begin{align}
\mathcal{M}_{c, h}^{(r)}\coloneqq {U}(\Vir)\otimes_{U(\Vir_{\geqslant 0})}\mathcal{D}_{S}.
\end{align}
Set $\Coh\coloneqq 1\otimes 1\in\mathcal{M}_{c, h}^{(r)}$.
We define the $\C$-grading of $\csm_{c, h}^{(r)}$ 
so that $\deg \Coh=h$. 
\end{definition}
\begin{proposition}
The pair $(\mathcal{M}_{c, h}^{(r)},\Coh)$ naturally equips with the structure of 
conformal coherent state $\Vir_c$-module. 
\end{proposition}
\begin{proof}
By construction, $\mathcal{M}_{c, h}^{(r)}$
is a $\Vir$-module with the property 
 that for any section $s\in \mathcal{M}_{c, h}^{(r)}$ 
there exists $N>0$ such that $[L_n, s]=0$ $(n>N)$.
It follows from this fact that $\mathcal{M}_{c, h}^{(r)}$ is a smooth $\widetilde{U}(\Vir_c)$-module, 
which implies the conditions (1), (2), (3) in Definition \ref{csm} 
(see Remark \ref{csmrem} and \cite[Section 5.1.8]{FBZ}).

Again by the construction, 
we have the expression
\begin{align}\label{qcM}
\mathcal{M}_{c, h}^{(r)}=
\bigoplus_{\substack{m_1,\dots, m_r\in\Z_{\geq 0}\\ n_1\geq n_2\geq \dots \geq n_k>0}}
\mathcal{O}_S L_{-n_1}\cdots L_{-n_k}\otimes \partial_{\lambda_1}^{m_1}\cdots \partial_{\lambda_r}^{m_r}\Coh.
\end{align}
This expression implies $(4)$. 
Since  $\partial_{\lam_r}$ is not in $(\csm_{c,h}^{(r)})_\O$, 
the quotient
$\csm_{c,h}^{(r)}/(\csm_{c,h}^{(r)})_\O$ is non-empty. 
The (proof of) Lemma 4.3 in \cite{I}  implies that the singularity of
$\mathcal{M}_{c, h}^{(r)}$ is  $H=\{\lambda_r=0\}$. 
 Hence we obtain $(5)$.
 The conformality is trivial by construction.
\end{proof}
\begin{remark}\label{simple}
At each point $\lam^o=(\lambda_1^o,\dots, \lam_r^o)$
with $\lam_r^o\neq 0$, 
the fiber $\csm_{c, h}^{(r)}|_{\lam^o}$ of $\csm_{c, h}^{(r)}$
have PBW basis
\begin{align*}
\csm_{c, h}^{(r)}|_{\lam^0}=\bigoplus_{\bm{n}\in \mathcal{P}^r}\C L_{\bm{n}}\coh{\lam^o}, 
\end{align*}
where $ \mathcal{P}^r$ denote the set of 
non-decreasing finite sequence 
\begin{align*}
&\bm{n}=(n_1,\dots, n_\ell), &n_1\leqslant \cdots \leqslant n_\ell< r,
\end{align*}
of integers and $L_{\bm{n}}=L_{n_1}\cdots L_{n_\ell}$
(we also set $L_\emptyset=1$). 
Hence $\csm_{c, h}^{(r)}|_{\lam^0}$
is a universal Whittaker module of type $h_{r}(\lam^o),\dots, h_{2r}(\lam^o)$
in the sense of \cite[Definition 2.2]{FJK}. 
In particular, $\csm_{c, h}^{(r)}|_{\lam^o}$ is 
a simple module (\cite[Corollary 2.2]{FJK}, see also \cite[Theorem 7]{LGZ} and \cite[Remark 2.2]{NS}).
\end{remark}

\section{Irregular vertex operator algebras}
In this section, we introduce the notion of irregular vertex algebras. 
\subsection{Notations on pullbacks}\label{Notation}
We shall fix the notation for pull backs of $\mathcal{D}$-modules 
over the space of internal parameters.
Let $S=\Spec(\C[\lam_j]_{j \in J})$ be a space of internal parameters.
To specify coordinate functions on $S$, 
we use the notation $S_\lambda=\Spec(\C[\lam_j]_{j \in J})$, 
$\mathcal{O}_{S_{\lam}}\coloneqq \C[\lambda_j]_{j\in J}$, and 
$\mathcal{D}_{S_\lam}=\C\langle \lambda_j,\partial_{\lambda_j}\rangle_{j\in J}$ 
instead of $S$, $\mathcal{O}_S$, and $\mathcal{D}_{S}$. 
In the rest of this paper, we often use the direct product of two or three copies of  $S$. To distinguish 
each component together with specified coordinate functions in the direct product, 
we write 
\[
S_{\lam} \times S_{\mu} \times S_{\nu}= \Spec (\C[\lam_j, \mu_j, \nu_j]_{j \in J})
\]
instead of $S^3 = \Spec(\C[\lam_j]_{j \in J}^{\otimes 3})$. 
We also use the notation $S_{\lambda,\mu}^2\coloneqq S_\lambda\times S_\mu$ 
and $S_{\lam,\mu,\nu}^3 \coloneqq S_{\lam} \times S_{\mu} \times S_{\nu}$. 

Let $\sigma_{\lambda+\mu} \colon S^2_{\lambda,\mu}\to S_\xi$ be the addition.
For a $\mathcal{D}_{S_\xi}$-module $\csm$,  
we denote the pull back of $\csm$ w.r.t. $\sigma_{\lambda+\mu}$ by 
$\csm_{\lam+\mu}\coloneqq\sigma_{\lam+\mu}^*\csm$
. 
Here we can explicitly write as  
\begin{align*}
\csm_{\lam+\mu}=\mathcal{O}_{S^2_{\lam,\mu}}\otimes_{\mathcal{O}_{S_{\xi}}}\csm
\end{align*}
where the tensor product is given through the morphism 
\begin{align*}
\O_{S_\xi}\to \O_{S^2_{\lam,\mu}},\quad \xi_j\mapsto \lam_j+\mu_j.
\end{align*}
The action of $\mathcal{D}_{S^2_{\lam,\mu}}=\C\langle \lambda_j,\partial_{\lambda_j}, 
\mu_j,\partial_{\mu_j}\rangle_{j\in J}$ on $\csm_{\lam+\mu}$ is  given by
\begin{align*}
&\partial_{\lam_j}(\psi(\lam,\mu)\otimes s)
=[\partial_{\lam_j}, \psi(\lam,\mu)]\otimes s+\psi(\lam,\mu)\otimes \partial_{\xi_j}s\\
&\partial_{\mu_j}(\psi(\lam,\mu)\otimes s)=[\partial_{\mu_j}, \psi(\lam,\mu)]\otimes s 
+\psi(\lam,\mu)\otimes \partial_{\xi_j}s.
\end{align*}
Since $\mathcal{D}_{S_\lam}$ and $\mathcal{D}_{S_\mu}$ are subalgebras of $\mathcal{D}_{S^2_{\lam,\mu}}$, 
the module $\csm_{\lam+\mu}$ naturally has 
both  $\mathcal{D}_{S_\lam}$-module and $\mathcal{D}_{S_\mu}$-module structures.
Similarly, we can define $\mathcal{D}_{S^3_{\lam,\mu,\nu}}$-module $\csm_{\lam+\mu + \nu}$ by using  
$\sigma_{\lambda+\mu+\nu} \colon S^3_{\lambda,\mu,\nu}\to S_\xi$.


\subsection{Filtered small lattices and exponential twists}\label{filter}
Let $S$ be a space of internal parameters. Let $V$ be a vertex algebra.
Let $\csm$ be a coherent state $V$-module over $S$.
As discussed in Section \ref{S csm}, 
$\csm$ has the singular divisor $H=\mathrm{Supp}(\csm/\csm_\mathcal{O})$.

We shall consider the localization of $\csm$ along $H$. 
In other words,  we consider $\csm(*H)\coloneqq \csm\otimes\mathcal{O}_S(*H)$, where 
$\mathcal{O}_S(*H)$ denotes the ring of rational functions on $S$
with poles in $H$.
By definition, we have $\csm(*H)=\csm_\mathcal{O}\otimes \mathcal{O}_S(*H)$.

\begin{definition}\label{FSL}
A \textit{filtered small lattice}  of $\csm(*H)$
is a pair $(\csm^\circ, F^\bullet)$
of a $\mathcal{D}_S$-submodule 
$\csm^\circ\subset \csm(*H)$ 
and a decreasing filtration $F^\bullet(\csm^\circ)$ of $\mathcal{O}_S$-submodules indexed by $\Z$
with the following properties:
\begin{itemize}
\item[(L)] We have $\csm\subset \csm^\circ\subset \csm(*H)$ and hence $\csm^\circ(*H)=\csm(*H)$.
\item[(F1)] There exists an integer $N$ such that we have $F^N(\csm^\circ)=\csm^\circ$.
\item[(F2)] We have $\bigcap_{m\in\Z}F^m(\csm^\circ)=0$.
\item[(F3)] For $k, \ell\in \Z$ and $m\in\C$, we have 
$\mathcal{D}_{S, k}\cdot F^\ell(\csm^\circ_m)\subset F^{\ell-k}(\csm^\circ_{m+k})$.
\item[(F4)] For any $A\in V$, $n\in \Z$, $A_{(n)}^\csm F^k(\csm^\circ)\subset F^k(\csm^\circ)$.
\end{itemize}
\end{definition}

Let $\overline{\csm}^\circ$ denote the completion of $U^\circ$ with respect to $F^\bullet (U^\circ)$
in the category of $\Z$-graded $\mathcal{O}_S$-modules.
We have a natural isomorphism 
$\overline{\csm}^\circ_k\simeq \prod_{\ell\in\Z}\mathrm{Gr}_F^\ell(\csm^\circ_k)$,
where $\mathrm{Gr}_F^\ell(\csm^\circ_k)\coloneqq F^\ell(\csm_k^\circ)/F^{\ell+1}(\csm_k^\circ)$.
By condition (F3), $\overline{\csm}^\circ$ naturally equips with the 
structure of $\mathcal{D}_S$-module.
By condition (F4), $A_{(n)}^\csm$ acts on $\csm^\circ$ and 
$\overline{\csm}^\circ$. However, we do not assume 
that $\csm^\circ$ (or, $\overline{\csm}^\circ$) is a coherent state $V$-module.

A coherent state module $\csm$ is called \textit{small} if it admits a filtered small lattice.
In this section, we assume that $\csm$ is small and fix a filtered small lattice $(\csm^\circ, F^\bullet)$.
We note that $\overline{\csm}_{\lambda+\mu}^\circ$
is the completion of $\csm^\circ_{\lambda+\mu}$ with respect
to the filtration
$F^\bullet(\csm^\circ_{\lambda+\mu})\coloneqq \sigma_{\lambda+\mu}^*F^\bullet(\csm^\circ)$.

We shall define the product
\begin{align}
\label{exp prod}
\mathcal{O}_{S_{\lambda,\mu}^2}\jump{z^{-1}}_0\otimes \csm^\circ_{\lambda+\mu}\pole{z}
\longrightarrow \overline{\csm}^\circ_{\lambda+\mu}\jump{z^{\pm 1}}
\end{align}
as follows: 
Take 
a power series 
\begin{align*}
f(z;\lambda,\mu)=\sum_{n\geq 0}f_n(\lambda,\mu)z^{-n}
\in \mathcal{O}_{S^2_{\lambda,\mu}}\jump{z^{-1}}_0
\end{align*}
with $f_n(\lambda,\mu)\in \mathcal{O}_{S^2_{\lambda,\mu}, -n}$,
and the homogeneous series 
\begin{align*}
u(z;\lambda,\mu)=\sum_{m\geq p}u_m(\lambda,\mu)z^m\in \csm_{\lambda+\mu}^\circ\pole{z}_d
\end{align*}
with $p\in\Z, d=\deg u(z;\lambda,\mu)\in \C$, $u_m(\lambda,\mu)\in \csm^\circ_{\lambda+\mu, m+d}$.

Take the biggest integer $N$ such that $F^N\csm^\circ=\csm^\circ$ 
(such $N$ exists by (F1) and (F2) in Definition \ref{FSL}).
Then, by the condition (F3), 
$f_n(\lambda;\mu)u_m(\lambda,\mu)$
is in $F^{N+n}(\csm^\circ_{\lambda+\mu, m-n})$.
It follows that 
the infinite sum
\begin{align*}
c_k(\lambda,\mu)\coloneqq \sum_{\substack{m-n=k\\  n\geq 0, m\geq p}}f_n(\lambda,\mu)u_m(\lambda,\mu)
\end{align*}
converges in $\overline{\csm}_{\lambda+\mu}^\circ$.
Hence, we can define (\ref{exp prod}) by
\begin{align}\label{elements}
f(z;\lambda,\mu)\otimes u(z;\lambda,\mu)\mapsto \sum_{k\in \Z}c_k(\lambda, \mu)z^k.
\end{align}
Here,  we note that 
$\csm^\circ_{\lambda+\mu}\pole{z}
=\bigoplus_{d\in\C}\csm^\circ_{\lambda+\mu}\pole{z}_d$
in our notation (see Section \ref{notation}).

Let $z^{-1}\mathcal{O}_{S^2}[z^{-1}]_0$ denote the ring of degree zero 
sections of $z^{-1}\mathcal{O}_{S^2}[z^{-1}]$.
Let $\mathcal{O}_{S^2}\jump{z^{-1}}^\times_0$ denote the abelian group
of degree zero invertible elements in $\mathcal{O}_{S^2}\jump{z^{-1}}$.
Then the map 
\begin{align*}
e^\bullet \colon z^{-1}\mathcal{O}_{S^2}[z^{-1}]_0\longrightarrow \mathcal{O}_{S^2}\jump{z^{-1}}_0^\times,\quad
\irr(z;\lambda,\mu)\mapsto e^{\irr(z;\lambda,\mu)}\coloneqq \sum_{n\geq 0}\frac{1}{n!}\irr(z;\lambda,\mu)^n
\end{align*}
defines a morphism of abelian groups.

\begin{definition}\label{exp twist}
For $\irr(z;\lambda,\mu)\in z^{-1}\mathcal{O}_{S^2}[z^{-1}]_0$,
let $e^{\irr(z;\lambda, \mu)}\csm^\circ_{\lambda+\mu}\pole{z}$ 
denote the image of 
\begin{align*}
\{e^{\irr(z;\lambda, \mu)} \otimes u(z)\mid u(z)\in \csm^\circ_{\lambda+\mu}\pole{z}\}
\end{align*}
by the morphism (\ref{exp prod}) (see also (\ref{elements})). 
\end{definition}

\subsection{Irregular fields and their compositions}

\begin{definition}[Irregular fields]\label{irr fields}
Let $\csm_\mu$ be a coherent state module on $S_\mu$ with  a
filtered small lattice $(\csm_\mu^{\circ},F^\bullet)$.
An \textit{irregular field with an irregularity $\irr(z;\lambda,\mu)\in z^{-1}\mathcal{O}_{S^2_{\lam,\mu}}[z^{-1}]_0$} on 
$\csm_\mu$ (with respect to $(\csm_\mu^\circ, F^\bullet)$) is an element 
$\mathcal{A}_\lambda(z)$ in 
\[
\Hom_{\mathcal{D}_{S_\mu}}
\(\csm^\circ_\mu, \overline{\csm}_{\lambda+\mu}^\circ\)
\jump{z^{\pm 1}}
\]
with the following property: 
for any $\mathcal{B}_\mu\in \csm_\mu^\circ$, 
we have 
\begin{align*}
\mathcal{A}_\lambda(z)\mathcal{B}_\mu\in e^{\irr(z;\lambda, \mu)}\csm^\circ_{\lambda+\mu}\pole{z}.
\end{align*}
\end{definition}

\begin{remark}\label{irregular field remark}
We also call an element 
\begin{align*}
\mathcal{A}_{\lambda+\mu}(z)\in 
\Hom_{\mathcal{D}_{S_\nu}}(\csm^\circ_\nu, \overline{\csm^\circ}_{\lambda+\mu+\nu})\jump{z^{\pm 1}}
\end{align*}
irregular field with an irregularity $\irr\(z;\lambda+\mu,\nu\)$
if we have 
\begin{align*}
\mathcal{A}_{\lambda+\mu}(z)\mathcal{C}_\nu\in e^{\irr\(z;\lambda+\mu,\nu\)}\csm^\circ_{\lambda+\mu+\nu}\pole{z}
\end{align*}
for any $\mathcal{C}_\nu\in \csm^\circ_\nu$.
We can also generalize the notion of irregular fields in a similar way.
\end{remark}

Similarly to Section \ref{filter}, 
we can naturally define the product 
\begin{align}\label{3 prod}
\mathcal{O}_{S^3}\jump{z^{-1},w^{-1}}_0\otimes 
\csm^\circ_{\lambda+\mu+\nu}\pole{z}\pole{w}\longrightarrow 
\overline{\csm}^\circ_{\lambda+\mu+\nu}\jump{z^{\pm 1}, w^{\pm 1}}.
\end{align}

Then, we define
$e^{\irr(z;\lambda,\mu+\nu)+\irr(w;\mu,\nu)}\csm^\circ\pole{z}\pole{w}$,
$e^{\irr(z;\lambda,\nu)+\irr(w;\mu,\nu)}\csm^\circ\pole{z}\pole{w}$, and so on 
in a way similar to Definition \ref{exp twist}.

\begin{definition}\label{composition}
For irregular fields $\mathcal{A}_\lambda(z)$ 
and $\mathcal{B}_\mu(w)$ with irregularity $\irr$, 
we define the composition
$\mathcal{A}_\lambda(z)\mathcal{B}_\mu(w)$,
which is an element in 
\begin{align*}
\Hom_{\mathcal{D}_{S_\nu}}(\csm^\circ_{\nu}, \overline{\csm}^\circ_{\lambda+\mu+\nu})\jump{z^{\pm 1}, w^{\pm 1}},
\end{align*}
as follows: 
Since $\mathcal{A}_\lambda(z)$ and $\mathcal{B}_\mu(w)$ are irregular fields, 
we have expansions
\begin{align*}
&\mathcal{A}_\lambda(z)=e^{\irr(z;\lambda,\eta)}\sum_{m\in \Z}\mathcal{A}'_{\lambda}(\eta)_{(m)}z^{-m-1},
&\mathcal{B}_\mu(w)=e^{\irr(w;\mu,\nu)}\sum_{n\in \Z}\mathcal{B}'_{\mu}(\nu)_{(n)}w^{-n-1}
\end{align*}
where $\mathcal{A}'_{\lambda}(\eta)_{(m)}$ $(m\in\Z)$ and $\mathcal{B}'_{\mu}(\nu)_{(n)}$ $(n\in\Z)$
are in 
$\Hom_{ \mathcal{O}_{S_\eta}  }(\csm^\circ_{\eta},\csm^\circ_{\lambda+\eta})$ and 
$\Hom_{\mathcal{O}_{S_\nu } }(\csm^\circ_{\nu},\csm^\circ_{\mu+\nu})$, respectively.
We then define 
\begin{align*}
\mathcal{A}_\lambda(z)\mathcal{B}_\mu(w)
\coloneqq e^{\irr(z;\lambda,\mu+\nu)+\irr(w;\mu,\nu)}
\sum_{m,n\in \Z}\mathcal{A}_\lambda'(\mu+\nu)_{(m)}\mathcal{B}_{\mu}'(\nu)_{(n)}z^{-m-1}w^{-n-1}.
\end{align*}
\end{definition}
We can easily check that the composition $\mathcal{A}_\lambda(z)\mathcal{B}_\mu(w)$ is 
independent of the parameters $\nu_i$. 
For each $\mathcal{C}_\nu\in \csm^\circ_\nu$, we have \begin{align*}
\mathcal{A}_\lambda(z)\mathcal{B}_\mu(w)\mathcal{C}_\nu
\in
e^{\irr(z;\lambda,\mu+\nu)+\irr(w;\mu,\nu)}\csm_{\lam+\mu+\nu}^\circ\pole{z}\pole{w}
\subset 
\overline{\csm}_{\lambda+\mu+\nu}^\circ\jump{z^{\pm 1}, w^{\pm 1}}.
\end{align*}

\subsection{Exponentially twisted Lie bracket}\label{twist_Lie}

\begin{lemma}\label{exp lemma}
Let $e^{-\irr(z-w)}_{|z|>|w|}$ denote the expansion of 
$e^{-\irr(z-w)}\in\mathcal{O}_{S^2}\jump{(z-w)^{-1}}$ in $\mathcal{O}_{S^2}\jump{z^{-1}, w}$.
Then, we have
$e^{-\irr(z-w)}_{|z|>|w|}e^{\irr(z)}\in \mathcal{O}_{S^2}[z^{-1}]\jump{w}$.
\end{lemma}
\begin{proof}
A priori, 
$e^{-\irr(z-w)}_{|z|>|w|}e^{\irr(z)}$
is in $\mathcal{O}_{S^2}\jump{z^{-1}, w}$. 
Consider the Taylor expansion
\begin{align}\label{taylor}
e^{-\irr(z-w)}_{|z|>|w|}e^{\irr(z)}=\sum_{k=0}^\infty c_k(z)w^k,
\end{align}
where each coefficient $c_k(z)$ is the restriction of
\begin{align*}
\frac{1}{k!}\(\frac{\partial}{\partial w}\)^{k} e^{-\irr(z-w)}_{|z|>|w|}e^{\irr(z)}
\end{align*}
to $w=0$.
The left hand side of (\ref{taylor})
is $1$ when when we restrict it to $w=0$. 
Hence we have $c_0(z)=1$.
For general $k\in\Z_{\geq 0}$, we have   
\begin{align*}
\frac{1}{k!}\(\frac{\partial}{\partial w}\)^{k} e^{-\irr(z-w)}_{|z|>|w|}e^{\irr(z)}
=\frac{1}{(k-1)!}\(\frac{\partial}{\partial w}\)^{k-1}\(-\frac{1}{k}\frac{\partial \irr(z-w)}{\partial w}\)
e^{-\irr(z-w)}_{|z|>|w|}e^{\irr(z)}.
\end{align*}
Hence we obtain that $c_k(z)\in \mathcal{O}_{S^2}[z^{-1}]$ by the induction on $k$.
\end{proof}
\begin{definition}\label{irregularity}
For an element $\irr(z;\lambda,\mu)$ of $z^{-1}\mathcal{O}_{S^2}[z^{-1}]_0$, 
consider the following properties:
\begin{itemize}
\item (skew symmetry) $\irr(z; \lambda,\mu)=\irr(-z;\mu, \lambda)$.
\item (additivity) $\irr(z;\lambda+\mu, \nu)=\irr(z;\lambda,\nu)+\irr(z;\mu,\nu)$.
\end{itemize}
The set of sections of $z^{-1}\mathcal{O}_{S^2}[z^{-1}]_0$
with these properties is denoted by $\mathrm{Irr}(S)$. 
\end{definition}
\begin{corollary}
For two irregular fields $\mathcal{A}_\lambda(z)$  and $\mathcal{B}_\mu(w)$
with an irregularity $\irr$ in $\mathrm{Irr}(S)$ and an element $\mathcal{C}_\nu\in \csm_\nu$,
we have
\begin{align*}
e^{-\irr(z-w;\lambda,\mu)}_{|z|>|w|}
\mathcal{A}_\lambda(z)\mathcal{B}_\mu(w)\mathcal{C}_\nu
\in e^{\irr(z;\lambda,\nu)+\irr(w;\mu,\nu)} 
\csm_{\lambda+\mu+\nu}^\circ\pole{z}\pole{w}.
\end{align*} 
\end{corollary}

\begin{definition}[Exponentially twisted Lie bracket]
For two irregular fields $\mathcal{A}_\lambda(z)$,  and $\mathcal{B}_\mu(w)$
with an irregularity $\irr\in\mathrm{Irr}(S)$, 
the \textit{exponentially twisted Lie bracket}
\begin{align*}
[\mathcal{A}_\lambda(z), \mathcal{B}_\mu(w)]_\irr
\end{align*} 
is defined as
\begin{align*}
e^{-\irr(z-w;\lambda,\mu)}_{|z|>|w|}\mathcal{A}_\lambda(z)\mathcal{B}_\mu(w)-
e^{-\irr(z-w;\lambda,\mu)}_{|w|>|z|}\mathcal{B}_\mu(w)\mathcal{A}_\lambda(z).
\end{align*}
The  two irregular fields $\mathcal{A}_\lambda(z)$ 
and $\mathcal{B}_\mu(w)$
are called \textit{mutually $\irr$-local}
if there exists an integer $N$ 
such that
\begin{align*}
(z-w)^N [\mathcal{A}_\lambda(z), \mathcal{B}_\mu(w)]_\irr=0.
\end{align*}
\end{definition}
\begin{lemma}\label{cor of loc}
For mutually $\irr$-local irregular fields $\mathcal{A}_\lambda(z)$ and $\mathcal{B}_\mu(w)$, 
and for any $\mathcal{C}_\nu\in \csm^\circ_\nu$, 
the two elements
\begin{align*}
&e^{\irr(z-w;\lambda,\mu)}_{|z|>|w|}\mathcal{A}_\lambda(z)\mathcal{B}_\mu(w)\mathcal{C}_\nu
\in e^{\irr(z;\lambda,\nu)+\irr(w;\mu,\nu)}\csm^\circ_{\lambda+\mu+\nu}\pole{z}\pole{w}\\
&e^{\irr(z-w;\lambda,\mu)}_{|w|>|z|}\mathcal{B}_\mu(w)\mathcal{A}_\lambda(z)\mathcal{C}_\nu
\in e^{\irr(z;\lambda,\nu)+\irr(w;\mu,\nu)}\csm^\circ_{\lambda+\mu+\nu}\pole{w}\pole{z}
\end{align*}
are the expansions of the same element
in 
\begin{align*}
e^{\irr(z;\lambda,\nu)+\irr(w;\mu,\nu)}\csm^\circ_{\lambda+\mu+\nu}\jump{z, w}[z^{-1}, w^{-1}, (z-w)^{-1}]
\end{align*}
to their respective domains.
\end{lemma}

\subsection{Envelopes of vertex algebras}
Let $V$ be a vertex algebra and $S$ be a space of internal parameters.
Let $(\env, Y_\env, \Coh)$ be a $\Z$-graded coherent state $V$-module over $S$
with $\deg \Coh=0$.
We consider the morphism 
\begin{align}
\Psi \colon V\longrightarrow \env_\O,\quad A\mapsto  A^\env_{(-1)}\Coh, 
\end{align}
where $A^\env_{(-1)}$ is defined in (\ref{Y_U}).

Assume moreover that 
we have a filtered small lattice  $(\env^\circ, F^\bullet)$  of $\env$. 
Consider $\env_\O$ as submodules of $\env^\circ$. 
Let  $\env_\O|^\circ_0$ denote 
the fiber of $\env_\O$ at the origin as a 
submodule of $\env^\circ$. 
In other words, we put  
\begin{align*}
\env_\O|_{0}^\circ\coloneqq \env_\O/(\env_\O \cap \mathfrak{m}_{S,0}\env^\circ),
\end{align*}
where $\mathfrak{m}_{S,0}$ denote the maximal
ideal of $\O_S$ corresponding to the origin of $S$. 
We note that $\env_\O \cap \mathfrak{m}_{S,0}\env^\circ$
is a $V$-submodule of $\env_\O$. 
Hence $\env_\O|_{0}^\circ$ is equipped with the structure of $V$-module. 

\begin{definition}\label{def:envelope}
A coherent state $V$-module $(\env, Y_\env, \Coh)$ on $S$ 
together with the filtered small lattice $(\env^\circ, F^\bullet)$ is 
called an \textit{envelope of $V$} if
the morphism $\Psi$ is injective, 
and 
the composition 
\begin{align*}
\overline{\Psi}\colon V\overset{\Psi}{\rightarrowtail} \env_\O\twoheadrightarrow \env_\O|_0^\circ 
\end{align*}
of  
$\Psi$ and the quotient map is an isomorphism of $V$-modules.
\end{definition}
For an envelope $\env$ of $V$, we identify
$\env_\O|_0^\circ$ and $V$ via $\overline{\Psi}$.

\subsection{Definition of irregular vertex algebras}
For an envelope  \[\env=(\env, Y_\env, \Coh, (\env^\circ, F^\bullet))\] 
of $V$ on $S_\lam$ (resp. $S_\mu$), 
write $\env_\lam \coloneqq \env$ and  $\coh{\lambda}  \coloneqq \Coh$  
 (resp. $\env_\mu \coloneqq \env$ and $\coh{\mu} \coloneqq \Coh  $). 
 We also use the notation $\Psi_\lambda(v)$ instead of $\Psi(v)$ for $v\in V$ and so on. 
Set $\coh{\lam+\mu} \coloneqq \sigma_{\lam + \mu}^* \Coh \in \env_{\lam+\mu}$.

We note that 
an irregular field $\mathcal{A}_\lambda(z)$ on $\env$
with irregularity $\irr\in \mathrm{Irr}(S)$
defines an element 
in $\Hom(V, {\env}_\lambda^\circ)\jump{z^{\pm 1}}$
in the following way: 
For a vector $v\in V$ consider 
$\mathcal{A}_\lambda(z)\Psi_\mu(v)\in e^{\irr(z;\lambda,\mu)}\env_{\lambda+\mu}^\circ\pole{z}$, 
and take the restriction to $\mu=0$. 
Since $\irr(z;\lambda,\mu)$ is a degree zero element in $z^{-1}\mathcal{O}_{S^2}[z^{-1}]$
and satisfies the skew symmetry, we have $\irr(z;\lambda,0)=0$.
Hence we have $\mathcal{A}_\lambda(z)\Psi_\mu(v)|_{\mu=0}\in \env_\lambda^\circ\pole{z}$, 
which is denoted by $\mathcal{A}_\lambda(z)v$ for short. 
In particular, we can define $\mathcal{A}_\lambda(z)\vac$.

For an endomorphism $T^\env \in\End(\env^\circ)$ and an irregular field $\mathcal{A}_\lambda(z)$, 
we define the Lie bracket $[T^\env,\mathcal{A}_\lambda(z)]$ by
\begin{align*}
(\sigma_{\lambda+\mu}^*T^\env)\mathcal{A}_\lambda(z)-\mathcal{A}_\lambda(z) T^\env.
\end{align*}

Consider 
$\Hom_{\mathcal{D}_{S_\mu}}\(\env^\circ_{\mu},\overline{\env}^\circ_{\lambda+\mu}\)\jump{z^{\pm 1}}$
as a $\mathcal{D}_{S_\lam}$-module 
by the $\mathcal{D}_{S_\lam}$-module structure on $\overline{\env}^\circ_{\lambda+\mu}$ 
and  $[\partial_{\lam_j},z^n]=0$ for $n\in \Z$. 
We shall define the notion of irregular vertex algebra as follows:
\begin{definition}[Irregular vertex algebra]\label{IVA}
Let $V$ be a vertex algebra. 
Let $S$ be a space of internal parameters. 
An \textit{irregular vertex algebra for $V$ on $S$}
is a tuple $(\env, \ssc, \irr, T^\env)$ of an envelope 
$\env=(\env, Y_\env, \Coh ,(\env^\circ, F^\bullet))$ of 
$V$ on $S$, 
a grade preserving $\mathcal{D}_{S_\lam}$-module morphism 
\begin{align*}
\ssc \colon
\env_\lambda^\circ\longrightarrow
\Hom_{\mathcal{D}_{S_\mu}}\(\env^\circ_{\mu},\overline{\env}^\circ_{\lambda+\mu}\)\jump{z^{\pm 1}},
\end{align*}
an element $\irr\in \mathrm{Irr}(S)$, and an endomorphism 
$T^\env\in \End_{\mathcal{D}_S}(\env^\circ)_1$ 
with the following properties:
\begin{itemize}
\item (irregular field axiom)
         For every $\mathcal{A}_\lambda\in \env_\lambda^\circ$, 
         the series $\ssc(\mathcal{A}_\lambda, z)$
         is an irregular field with the irregularity $\irr(z;\lambda,\mu)$.
\item (irregular locality axiom)
         For any $\mathcal{A}_\lambda\in \env_\lambda^\circ$, $\mathcal{B}_\mu \in \env_\mu^\circ$,
         two irregular fields $\ssc(\mathcal{A}_\lambda, z)$ and $\ssc(\mathcal{B}_\mu, w)$
         are mutually $\irr$-local.
\item (vacuum axiom)
         For any $\mathcal{A}_\lambda\in \env_\lambda^\circ$, we have 
         $\ssc(\mathcal{A}_\lambda, z)\vac\in \env_\lambda^\circ\jump{z}$, 
         and $\ssc(\mathcal{A}_\lambda, z)\vac|_{z=0}=\mathcal{A}_\lambda$.
\item (coherent state axiom) 
         We have 
         $\ssc(\coh{\lambda}, z)\coh{\mu}\in e^{\irr(z;\lambda,\mu)}\env_{\lambda+\mu}^\circ\jump{z}$, 
         and 
         \begin{align*}
         e^{-\irr(z;\lambda,\mu)}\ssc(\coh{\lambda}, z)\coh{\mu}|_{z=0}=\coh{\lambda+\mu}.
         \end{align*}       
\item (translation axiom)
         We have $[T^\env, \ssc(\mathcal{A}_\lambda, z)]=\partial_z\ssc(\mathcal{A}_\lambda, z)$
         for any $\mathcal{A}_\lambda\in \env_\lambda^\circ$.
\item (compatibility condition) 
         For any ${A}\in V$, $\mathcal{B}_\mu\in \env_{\mu}$, 
         the restriction of $\ssc(\Psi_\lambda(A), z)\mathcal{B}_\mu$ to 
         $\lambda=0$ is $Y_\env(A, z)\mathcal{B}_\mu$.
         We also have \[T^\env(\Psi(A))|_{0}^\circ=\overline{\Psi}(TA),\] 
         where $*|_0^\circ$ denotes the restriction as a section of $\env^\circ$.
\end{itemize}
\end{definition}

\begin{remark}\label{RMK}
The endomorphism $T^\env$ will be denoted by $T$ if it is not confusing. 
The condition that $Y(\cdot, z)$ 
is a morphism of $\mathcal{D}_{S_\lam}$-modules 
and takes values in $\Hom_{\mathcal{D}_{S_\mu}}\(\env^\circ_{\mu},\overline{\env}^\circ_{\lambda+\mu}\)\jump{z^{\pm 1}}$ 
implies that for  $\mathcal{A}_\lambda \in \env^\circ$, 
\begin{align*}
[\partial_{\lam_j}, Y(\mathcal{A}_\lambda, z)]=Y(\partial_{\lam_j} \mathcal{A}_\lambda,z), 
\quad [\partial_{\mu_j},Y(\mathcal{A}_\lambda, z)]=0
\end{align*}
for any $j\in J$. 
Thus, Fourier coefficients $\mathcal{A}_{\lambda,(n)}$ of the irregular vertex operator 
$Y(\mathcal{A}_\lambda, z)=\sum_{n \in \Z} \mathcal{A}_{\lambda,(n)} z^{-n-1}$ are 
independent of the parameters $\mu_j$. 
A irregular vertex algebra whose coherent state module is non-singular 
is called a \textit{non-singular irregular vertex algebra}.
\end{remark}

Examples of irregular vertex algebras will be given in Section \ref{FFIVA} and Section \ref{ffr}.

\subsection{Conformal structures}

We shall define the notion of conformal structure on irregular vertex algebras
i.e. \textit{irregular vertex operator algebras}. 
Let $V$ be a vertex operator algebra, i.e.
a vertex algebra $V$ together with the conformal vector $\omega$ (see Section \ref{ccsm}). 
Let $S$ be a space of internal parameters with a conformal structure
\begin{align*}
\rho_S \colon \Der_0(\C\jump{t})\longrightarrow \Theta_S,\quad t^{j+1}\partial_t\mapsto -D_j.
\end{align*}
Define vector fields $D_j^{\lam} \in \Theta_{S_\lam}$ (resp. $D_j^\mu \in \Theta_{S_\mu}$)  
for $j=0,\dots,r-1$ as the images of $ -t^{j+1}\partial_t$ 
via the above map $\rho_{S_\lam}$ (resp.  $\rho_{S_\mu}$). 
We consider the action of $D_j^{\lam} $ and  $D_j^\mu$ on $\mathcal{O}_{S^2_{\lam, \mu}}$. 
Let $\irr(z;\lambda,\mu)$ be an irregularity on $S$. 
Since $\irr(z;\lambda,\mu)$ is degree zero, i.e. $\irr(z;\lambda,\mu)\in z^{-1}\mathcal{O}_{S^2_{\lam, \mu}}[z^{-1}]_0$, 
we have 
\begin{align*}
\(D_{0}^\lambda+D_{ 0}^\mu+z\partial_z\)\irr(z;\lambda,\mu)=0.
\end{align*}
\begin{definition}\label{conformal irregularity}
An irregularity $\irr(z;\lambda,\mu)$ is called \textit{conformal}
if it satisfies the differential equations
\begin{align}\label{m seq}
\(D_{ j}^\mu+\sum_{0\leq m\leq j}\(\partial_z^{(m)}z^{j+1}\)D_{ m}^\lambda+z^{j+1}\partial_z\)\irr(z;\lambda,\mu)=0
\mod \mathcal{O}_{S^2_{\lam,\mu}}\jump{z}
\end{align} 
for any non-negative integer $j$.
An irregular vertex algebra $\env$ is called an \textit{irregular vertex operator algebra} 
if $\env$ and $\irr$ are conformal and $T^\env=L_{-1}^\env$.
\end{definition}

\section{Associativity and operator product expansions}
In this section, we shall prove the three fundamental properties of irregular vertex algebras:
Goddard uniqueness theorem, associativity, and operator product expansions. 
The proofs are almost parallel to the classical ones under suitable formulations.
\subsection{Goddard Uniqueness theorem}
Let $(\env, (\env^\circ, F^\bullet),\ssc,\irr, {T})$ be an irregular vertex algebra
for a vertex algebra $V$ on a space $S$ of internal parameters.
The following is an analog of Goddard Uniqueness theorem:
\begin{theorem}\label{GUT}
Let $\mathcal{A}_\lambda(z)$ be an irregular field on $\env$ with the irregularity $\irr$. 
If 
\begin{enumerate}
\item for any $b_\mu\in \env_\mu^\circ$, irregular fields $\mathcal{A}_\lambda(z)$ and 
$\ssc(b_\mu, w)$ are mutually $\irr$-local, 
\item for an element $a_\lambda\in \env^\circ_\lam$, we have $\mathcal{A}_\lambda(z)\vac=\ssc(a_\lambda, z)\vac$,
\end{enumerate}
then we have $\mathcal{A}_\lambda(z)=\ssc(a_\lambda, z)$. 
\end{theorem}
\begin{proof}
Let $b_\mu$ be an element in $\env_\mu^\circ$. 
By the assumptions and the irregular locality axiom (Definition \ref{IVA} (2)), 
there exists a positive integer $N$
such that 
\begin{align*}
&(z-w)^Ne^{-\irr(z-w;\lambda,\mu)}_{|z|>|w|}\mathcal{A}_\lambda(z)\ssc(b_\mu, w)\vac\\
=&(z-w)^Ne^{-\irr(z-w;\lambda,\mu)}_{|w|>|z|}\ssc(b_\mu, w)\mathcal{A}_\lambda(z)\vac 
&(\text{assumption (1)})\\
=&(z-w)^N e^{-\irr(z-w;\lambda,\mu)}_{|w|>|z|}  \ssc(b_\mu, w)\ssc(a_\lambda, z)\vac
&(\text{assumption (2)})\\
=&(z-w)^N e^{-\irr(z-w;\lambda,\mu)}_{|z|>|w|}  \ssc(a_\lambda, z) \ssc(b_\mu, w)\vac
&(\text{Irregular locality axiom}).
\end{align*}
Therefore, we obtain 
\begin{align*}
(z-w)^Ne^{-\irr(z-w;\lambda,\mu)}_{|z|>|w|}\mathcal{A}_\lambda(z)\ssc(b_\mu, w)\vac
=
(z-w)^N e^{-\irr(z-w;\lambda,\mu)}_{|z|>|w|}  \ssc(a_\lambda, z) \ssc(b_\mu, w)\vac
\end{align*}
Since we can restrict both sides to $w=0$, we get
\begin{align*}
z^Ne^{-\irr(z;\lambda,\mu)}\ssc(a_\lambda, z)b_\mu=
z^Ne^{-\irr(z;\lambda,\mu)}\mathcal{A}_\lambda(z)b_\mu.
\end{align*}
This implies the theorem.
\end{proof}

\begin{lemma}\label{e^zT}
For any $\mathcal{A}\in \env^\circ$,
we have $\ssc(\mathcal{A}, z)\vac=e^{zT}\mathcal{A}$. 
\end{lemma}
\begin{proof}
Take the expansion 
$\ssc(\mathcal{A}, z)=\sum_{n\in\Z}\mathcal{A}_{(n)}z^{-n-1}$, 
$\mathcal{A}_{(n)}\in \Hom(\env_\mu^\circ, \overline{\env}_{\lambda+\mu}^\circ)$. 
By the vacuum axiom in Definition \ref{IVA}, 
we have $\mathcal{A}_{(n)}\vac=0 $ for $n\geq 0$, 
and $\mathcal{A}_{(-1)}\vac=\mathcal{A}$.  
By the translation axiom, we have
$\partial_z\ssc(\mathcal{A}, z)=T\mathcal{A}(z)\vac.$
Hence we obtain $n\mathcal{A}_{(-n-1)}\vac=T\mathcal{A}_{(-n)}$.
Therefore, we obtain $\mathcal{A}_{(-n-1)}\vac=(n!)^{-1}T^n\mathcal{A}$. 
This implies the lemma.
\end{proof}
\begin{corollary}
Assume that an irregular field
$\mathcal{A}_{\lambda}(z)\in 
\Hom(\env_\mu^\circ,\overline{\env}_{\lambda+\mu}^\circ)\jump{z^{\pm 1}}$
and $\ssc(\mathcal{B}_\mu, w)$ are mutually $\irr$-local for any $\mathcal{B}_\mu\in \env^\circ_\mu$, 
$$\mathcal{A}_\lambda(z)\vac-a_\lambda\in z\env_\lambda^\circ\jump{z}$$ for some 
$a_\lambda\in \env^\circ_\lambda$, 
and 
$\partial_z\mathcal{A}_\lambda(z)\vac=T\mathcal{A}_\lambda(z)\vac$. 
Then we obtain $\mathcal{A}_\lambda(z)=\ssc(a_\lambda, z)$.
\end{corollary}

\begin{lemma}\label{z+w}
For any $\mathcal{A}_\lambda\in \env^\circ_\lambda$, we have
\begin{align*}
e^{wT}\ssc(\mathcal{A}_\lambda)e^{-wT}=\ssc(\mathcal{A}_\lambda, z+w)
\end{align*}
in $\Hom(\env_\mu^\circ, \overline{\env}_{\lambda+\mu}^\circ)\jump{z^{\pm 1}}$, 
where $(z+w)^{-1}$ is expanded in $\C\pole{z}\pole{w}$.
\end{lemma}

\begin{proposition}[skew symmetry]\label{skew}
For $\mathcal{A}_\lambda\in \env_\lambda^\circ$, $\mathcal{B}_\mu\in \env_\mu^\circ$, we have
\begin{align*}
\ssc(\mathcal{A}_\lambda, z)\mathcal{B}_\mu=e^{zT}\ssc(\mathcal{B}_\mu, -z)\mathcal{A}_\lambda
\end{align*}
in $e^{\irr(z;\lambda,\mu)}\env_{\lambda+\mu}^\circ\pole{z}$.
\end{proposition}
\begin{proof}
For sufficiently large $N\in \Z$, we have
\begin{align*}
&(z-w)^Ne^{-\irr(z-w;\lambda,\mu)}_{|z|>|w|}\ssc(\mathcal{A}_\lambda, z)
\ssc(\mathcal{B}_\mu, w)\vac\\
=&(z-w)^Ne^{-\irr(z-w;\lambda, \mu)}_{|w|>|z|}\ssc(\mathcal{B}_\mu, w)
\ssc(\mathcal{A}_\lambda, z)\vac&(\irr\text{-locality})\\
=&(z-w)^Ne^{-\irr(z-w;\lambda, \mu)}_{|w|>|z|}\ssc(\mathcal{B}_\mu, w)
e^{zT}\mathcal{A}_\lambda&(\text{Lemma \ref{e^zT}})\\
=&(z-w)^Ne^{-\irr(z-w;\lambda,\mu)}_{|w|>|z|}e^{zT}\ssc(\mathcal{B}_\mu, w-z)\mathcal{A}_\lambda
&(\text{Lemma \ref{z+w}}).
\end{align*}
If we take $N$ sufficiently large, 
any term in these equalities are in $\env_{\lambda+\mu}^\circ\jump{z, w}$. 
Hence we can restrict them to $w=0$ and obtain
\begin{align*}
z^Ne^{-\irr(z;\lambda,\mu)}\ssc(\mathcal{A}_\lambda, z)\mathcal{B}_\mu
=z^Ne^{-\irr(z;\lambda,\mu)}e^{zT}\ssc(\mathcal{B}_\mu, -z)\mathcal{A}_\lambda.
\end{align*}
This implies the proposition.
\end{proof}
\subsection{Associativity}
The following theorem is a generalization of the associativity 
to irregular vertex algebras:
\begin{theorem}\label{associative}
For any $\mathcal{A}_\lambda\in \env^\circ_\lambda$, $\mathcal{B}_\mu\in \env^\circ_\mu$, 
and $\mathcal{C}_\nu\in  \env^\circ_\nu$,
the three elements
\begin{align*}
&e^{-\irr(z;\lambda, \nu)}e^{-\irr(z-w;\lambda,\mu)}_{|z|>|w|}
\ssc(\mathcal{A}_\lambda, z)\ssc(\mathcal{B}_\mu, w)\mathcal{C}_\nu
\in e^{\irr(w;\mu,\nu)}\env_{\lambda+\mu+\nu}^\circ\pole{z}\pole{w}\\
&e^{-\irr(z;\lambda,\nu)}e^{-\irr(z-w;\lambda,\mu)}_{|w|>|z|}
\ssc(\mathcal{B}_\mu, w) \ssc(\mathcal{A}_\lambda, z)\mathcal{C}_\nu
\in e^{\irr(w;\mu,\nu)}\env_{\lambda+\mu+\nu}^\circ\pole{w}\pole{z}\\
&e^{-\irr(z;\lambda,\nu)}_{|w|>|z-w|}e^{-\irr(z-w;\lambda,\mu)}
\ssc(\ssc(\mathcal{A}_\lambda, z-w)\mathcal{B}_\mu, w)\mathcal{C}_\nu
\in e^{\irr(w;\mu,\nu)}\env^\circ_{\lambda+\mu+\nu}\pole{w}\pole{z-w}
\end{align*} 
are the expansions of the same element in 
\begin{align*}
e^{\irr(w;\mu,\nu)}\env^\circ_{\lambda+\mu+\nu}\jump{z, w}[z^{-1}, w^{-1}, (z-w)^{-1}]
\end{align*}
to their respective domains.
\end{theorem}

Take an expansion 
\begin{align*}
e^{-\irr(z-w;\lambda,\mu)}\ssc(\mathcal{A}_\lambda, z-w)=\sum_{n\in\Z}\mathcal{A}_{\lambda}'(\mu)_{(n)}(z-w)^{-n-1},
\end{align*}
where $\mathcal{A}'_{\lambda}(\mu)_{(n)}\in \Hom(\env^\circ_\mu, \env^\circ_{\lambda+\mu})$
for each $n\in \Z$.
Since $e^{-\irr(z-w;\lambda,\mu)}\ssc(\mathcal{A}_\lambda, z-w)\mathcal{B}_\mu$
is in $\env^\circ_{\lambda+\mu}\pole{z-w}$,
we have the expansion
\begin{align*}
e^{-\irr(z-w;\lambda,\mu)}\ssc(\mathcal{A}_\lambda, z-w)\mathcal{B}_\mu
=
\sum_{n\leq N}\frac{\mathcal{A}'_{\lambda}(\mu)_{(n)}\mathcal{B}_\mu}{(z-w)^{n+1}}
\end{align*}
for sufficiently large $N$.

Then, the composition
$e^{-\irr(z;\lambda,\nu)}_{|w|>|z-w|}e^{-\irr(z-w;\lambda,\mu)}
\ssc(\ssc(\mathcal{A}_\lambda, z-w)\mathcal{B}_\mu, w)\mathcal{C}_\nu$
is defined as
\begin{align*}
e^{-\irr(z;\lambda,\nu)}_{|w|>|z-w|}
\sum_{n\leq N}\frac{\ssc(\mathcal{A}'_{\lambda}(\mu)_{(n)}\mathcal{B}_\mu, w)\mathcal{C}_\nu}{(z-w)^{n+1}}.
\end{align*}
Here, note that we have 
$e^{-\irr(z;\lambda,\nu)}_{|w|>|z-w|}e^{\irr(w;\lambda, \nu)}
\in \mathcal{O}_{S^2_{\lambda,\nu}}[w^{-1}]\jump{z-w}$
by Lemma \ref{exp lemma}.
Hence 
$e^{-\irr(z;\lambda,\nu)}_{|w|>|z-w|}\ssc(\mathcal{A}'_{\lambda}(\mu)_{(n)}\mathcal{B}_\mu, w)\mathcal{C}_\nu$
is in $e^{\irr(w;\mu,\nu)}\env^\circ_{\lambda+\mu+\nu}\pole{w}\jump{z-w}$.

\begin{proof}[Proof of Theorem \ref{associative}]
By the skew symmetry (Proposition \ref{skew}), 
we have 
\begin{align*}
&e^{-\irr(z-w;\lambda,\mu)}_{|z|>|w|}
\ssc(\mathcal{A}_\lambda, z)\ssc(\mathcal{B}_\mu, w)\mathcal{C}_\nu\\
=&e^{-\irr(z-w;\lambda,\mu)}_{|z|>|w|}
\ssc(\mathcal{A}_\lambda, z)
e^{wT}\ssc(\mathcal{C}_\nu, -w)\mathcal{B}_\mu\\
=&e^{-\irr(z-w;\lambda,\mu)}_{|z|>|w|}
e^{wT}\(e^{-wT}\ssc(\mathcal{A}_\lambda, z)e^{wT}\)\ssc(\mathcal{C}_\nu, -w)\mathcal{B}_\mu
\end{align*}
Since 
$e^{-\irr(z-w;\lambda,\mu)}_{|z|>|w|}
\ssc(\mathcal{A}_\lambda, z)
e^{wT}\ssc(\mathcal{C}_\nu, -w)\mathcal{B}_\mu$
is in $e^{\irr(z;\lambda,\nu)+\irr(w;\mu,\nu)}\env^\circ_{\lambda+\mu+\nu}\pole{z}\pole{w}$, 
the last equality makes sense. 
Then, again by Lemma \ref{z+w}, 
this equals to 
\begin{align*}
e^{-\irr(z-w;\lambda,\mu)}_{|z|>|w|}e^{wT}
\ssc(\mathcal{A}_\lambda,z-w)\ssc(\mathcal{C}_\nu, -w)\mathcal{B}_\mu. 
\end{align*}

Therefore, the
two elements
\begin{align}\notag
&e^{\irr(z;\lambda,\nu)}e^{-\irr(z-w;\lambda,\mu)}_{|z|>|w|}
\ssc(\mathcal{A}_\lambda, z)\ssc(\mathcal{B}_\mu, w)\mathcal{C}_\nu{\text{ and}}\\\label{|z-w|>|w|}
&e^{\irr(z;\lambda,\nu)}_{|z-w|>|w|}e^{-\irr(z-w;\lambda,\mu)}e^{wT}
\ssc(\mathcal{A}_\lambda,z-w)\ssc(\mathcal{C}_\nu, -w)\mathcal{B}_\mu
\end{align}
are the expansions of the same element in 
\begin{align*}
e^{\irr(w;\mu,\nu)}\env^\circ_{\lambda+\mu+\nu}\jump{z, w}[z^{-1}, w^{-1}, (z-w)^{-1}]
\end{align*}
to the modules
$e^{\irr(w;\mu,\nu)}\env^\circ_{\lambda+\mu+\nu}\pole{z}\pole{w}$ and 
$e^{\irr(w;\mu,\nu)}\env^\circ_{\lambda+\mu+\nu}\pole{z-w}\pole{w}$ respectively.

By the skew symmetry, we have 
\begin{align*}
\ssc(\mathcal{A}'_{\lambda}(\mu)_{(n)}\mathcal{B}_\mu, w)\mathcal{C}_\nu
=
e^{wT}\ssc(\mathcal{C}_\nu,-w)\mathcal{A}'_{\lambda}(\mu)_{(n)}\mathcal{B}_\mu.
\end{align*}
Hence, we obtain 
\begin{align}\notag
&e^{\irr(z;\lambda,\nu)}_{|w|>|z-w|}e^{-\irr(z-w;\lambda,\mu)}
\ssc(\ssc(\mathcal{A}_\lambda, z-w)\mathcal{B}_\mu, w)\mathcal{C}_\nu\\\label{|w|>|z-w|}
=&e^{\irr(z;\lambda,\nu)}_{|w|>|z-w|}e^{-\irr(z-w;\lambda,\mu)}e^{wT}\ssc(\mathcal{C}_\nu,-w)
\ssc(\mathcal{A}_\lambda, z-w)\mathcal{B}_\mu
\end{align}
in $\overline{\env}_{\lambda+\mu+\nu}\jump{w^{\pm 1}, (z-w)^{\pm 1}}$.
By Lemma \ref{cor of loc}, 
(\ref{|z-w|>|w|}) and (\ref{|w|>|z-w|})
are the expansion of the same element.
This proves the theorem.
\end{proof}

\subsection{Normally ordered product and operator product expansion}
We shall define the normally ordered product for irregular fields:
\begin{definition}[normally ordered product]\label{NOP}
For an irregular field 
$\mathcal{A}_\lambda(z)$
with an expansion 
$\mathcal{A}_\lambda(z)=
e^{\irr(z;\lambda, \nu)}
\sum_{n\in\Z}\mathcal{A}'_{\lambda}(\nu)_{(n)}z^{-n-1}$,
set 
\[
\mathcal{A}_\lambda'(z;\nu)\coloneqq e^{-\irr(z;\lambda,\nu)}\mathcal{A}_\lambda(z)=
\sum_{n \in \Z}\mathcal{A}'_{\lambda}(\nu)_{(n)}z^{-n-1}
\]
and
\begin{align*}
\mathcal{A}'_\lambda(z;\nu)_+\coloneqq \sum_{n<0}\mathcal{A}'_{\lambda}(\nu)_{(n)}z^{-n-1}, \quad
\mathcal{A}'_\lambda(z;\nu)_-\coloneqq \sum_{n\geq 0}\mathcal{A}'_{\lambda}(\nu)_{(n)}z^{-n-1}.
\end{align*}
Let $\mathcal{B}_\mu(w)$ be an irregular field with irregularity $\irr(w;\mu, \nu)$ and 
set $\mathcal{B}_\mu'(w;\nu) \coloneqq  e^{-\irr(w;\mu,\nu)}\mathcal{B}_\mu(w)$. 
The \textit{normally ordered product }
$\nop{\mathcal{A}_\lambda(z)\mathcal{B}_\mu(w)}$ of $\mathcal{A}_\lambda(z)$ 
and $\mathcal{B}_\mu(w)$  is defined by
\begin{align*}
\nop{\mathcal{A}_\lambda(z)\mathcal{B}_\mu(w)}
\coloneqq 
e^{\irr(z;\lambda,\nu)+\irr(w;\mu,\nu)}\(
\mathcal{A}'_\lambda(z;\mu+\nu)_+\mathcal{B}'_\mu(w;\nu)
+\mathcal{B}'_\mu(w;\lambda+\nu)\mathcal{A}'_\lambda(z;\nu)_-\).
\end{align*}

The restriction of $\nop{\mathcal{A}_\lambda(z)\mathcal{B}_\mu(w)}$ 
to $z=w$ is well defined and is denoted by $\nop{\mathcal{A}_\lambda(z)\mathcal{B}_\mu(z)}$.
Note that $\nop{\mathcal{A}_\lambda(z)\mathcal{B}_\mu(z)}$
is again an irregular field with irregularity $\irr(z;\lambda+\mu, \nu)$. Actually, 
we can check that $\nop{\mathcal{A}_\lambda(z)\mathcal{B}_\mu(z)}$ does not depend on
the parameters $\nu_i$ by using the presentation in Lemma \ref{lem:delta} below. 
\end{definition}
The following two lemmas can be proved 
by the same way as the classical case:
\begin{lemma}
\label{lem:delta}
The restriction  $\nop{\mathcal{A}_\lambda(w)\mathcal{B}_\mu(w)}$ equals to
$e^{\irr(w;\lambda+\mu,\nu)}$ times
\begin{align*}
\mathrm{Res}_{z=0}
\left[\delta(z-w)_-\mathcal{A}'_\lambda(z;\mu+\nu)\mathcal{B}'_\mu(w;\nu)
+\delta(z-w)_{+}\mathcal{B}'_\mu(w;\lambda+\nu)\mathcal{A}'_{\lambda}(z;\nu)\right]dz,
\end{align*}
where $\delta(z-w)_-\coloneqq \sum_{n=0}^\infty w^n/z^{n+1}$ and 
$\delta(z-w)_+\coloneqq \sum_{n>0}z^{n-1}/w^n$.
\end{lemma}

\begin{lemma}[Dong's lemma]\label{Dong}
Let $\mathcal{A}_\lambda(z), \mathcal{B}_\mu(w), \mathcal{C}_\nu(u)$ 
be irregular fields with an irregularity $\irr$. 
Assume that each two of three fields are mutually $\irr$-local.
Then, 
the normally ordered product $\nop{\mathcal{A}_\lambda(z)\mathcal{B}_\mu(z)}$
and $\mathcal{C}_\nu(w)$ are mutually $\irr$-local.
\end{lemma}

\begin{theorem}[Operator product expansion]\label{ope theorem}
For any $\mathcal{A}_{\lambda}\in \env_\lambda^\circ$ and $\mathcal{B}_\mu\in \env^\circ_\mu$, there 
is an equality
\begin{align*}
\ssc(\mathcal{A}_\lambda, z)\ssc(\mathcal{B}_\mu, w)
=e^{\irr(z-w;\lambda,\mu)}\left(
\sum_{n=0}^\infty \frac{\ssc(\mathcal{A}'_{\lambda}(\mu)_{(n)}\mathcal{B}_\mu, w)}{(z-w)^{n+1}}
+\nop{\ssc(\mathcal{A}_\lambda, z)\ssc(\mathcal{B}_\mu, w)}\right)
\end{align*}
where 
$\ssc(\mathcal{A}_\lambda, z)
=e^{\irr(z;\lambda,\mu)}\sum_{n\in \Z}\mathcal{A}'_{\lambda}(\mu)_{(n)}z^{-n-1}$ 
and both sides are expanded in
the domain $|z|>|w|$. 
\end{theorem}
\begin{proof}
By the associativity, 
it remains to show that 
\begin{align}\label{nop rep}
\ssc(\mathcal{A}'_{\lambda}(\mu)_{(-n-1)}\mathcal{B}_\mu, w)
=
\nop{\(\partial_w^{(n)}\ssc(\mathcal{A}_\lambda, w)\)\ssc(\mathcal{B}_\mu, w)}
\end{align}
for every non-negative integer $n$.
By the direct computation, 
we have 
$$\nop{\(\partial_w^{(n)}\ssc(\mathcal{A}_\lambda, w)\)\ssc(\mathcal{B}_\mu, w)}\vac|_{w=0}
=\mathcal{A}'_{\lambda}(\mu)_{(-n-1)}\mathcal{B}_\mu.$$
The irregular field $\nop{\(\partial_w^{(n)}\ssc(\mathcal{A}_\lambda, w)\)\ssc(\mathcal{B}_\mu, w)}$
and $\ssc(\mathcal{C}_\nu, z)$ are mutually $\irr$-local for every $\mathcal{C}_\nu$
by the Dong's lemma (Lemma \ref{Dong}).
Then, by the Goddard uniqueness theorem (Theorem \ref{GUT}),
we obtain (\ref{nop rep}).
\end{proof}
As an easy consequence, 
we obtain the following:
\begin{corollary} 
The composition 
$e^{-\irr(z-w;\lambda,\mu)}_{|z|>|w|}\ssc(\coh{\lambda}, z)\ssc(\coh{\mu}, w)$
can be restricted to $z=w$. Moreover, the restriction equals
to 
$
\nop{\ssc(\coh{\lambda}, z)\ssc(\coh{\mu}, z)}
=\ssc(\coh{\lambda+\mu},z).
$
\end{corollary}

\section{Irregular Heisenberg vertex operator algebras}\label{FFIVA}
In this section, following the ideas of \cite{NS}, we shall define the irregular vertex algebras
for the Heisenberg vertex algebra.

\subsection{Heisenberg vertex algebra}
Let us briefly recall the definition of Heisenberg vertex algebra to fix the notation.
Let ${\Heis}$ denote the $\Z_{\geq 0}$-graded vector space
of graded polynomial ring $\C[x_{n}]_{n> 0}$ of variables $x_{n}$ with $\deg x_{n}=n\in \Z_{>0}$. 
Let 
$\vac\in {\Heis}$ denote the unit of $\C[x_{n}]_{n> 0}$. 

Define an endomorphism $T$ as the derivation on $\C[x_{n}]_{n> 0}$
with $Tx_{n}=n x_{n+1}$.  
Fix a non-zero complex number $\kappa$. 
Let 
$a_{-n}$ (resp.  $a_{n}$) denote the multiplication of $x_{n}$, 
(resp.  the derivation $2\kappa n{\partial_{x_{n}}}$) for $n>0$. 
Set $a_{0}\coloneqq 0\in\End({\Heis})$.
The power series 
$a(z)=\sum_{n\in \Z}a_{n}z^{-n-1}$ defines an field on ${\Heis}$.

Define $Y(\cdot, z) \colon {\Heis}\to \End({\Heis})\jump{z^{\pm 1}}$ by
\begin{align}\label{field Heis.}
Y(a_{-n_1}a_{-n_2}\cdots a_{-n_k}\vac, z)
\coloneqq \nop{\partial_z^{(n_1-1)}a(z)\cdots \partial_z^{(n_k-1)}a(z)}
\end{align}
for $n_1,\dots,n_k\in \Z_{>0}$, $k>0$. We also set $Y(\vac, z)=\id$.
Then, the tuple 
\begin{align*}
{\Heis}_\kappa\coloneqq ({\Heis},\vac, T,  Y(\cdot, z))
\end{align*}
is known to be a vertex algebra called \textit{Heisenberg vertex algebra}.

\subsection{Coherent state module}\label{csm-H}
Fix a positive integer $r$.
Let $S\coloneqq \Spec(\C[\lambda_j]_{j=1}^r)$
be the space of internal parameters with $\deg \lambda_j=-j$. 

Consider the completion $\overline{{\Heis}}\coloneqq \prod_{n\geq 0}{\Heis}_n$, 
where ${\Heis}_n$ is the degree $n$-part of ${\Heis}$.
The tensor product $\overline{{\Heis}}_S\coloneqq \overline{{\Heis}}\otimes_\C\mathcal{O}_{S}$
is a $\Z$-graded $\mathcal{O}_S$-module
whose degree $n$-part is given by 
\begin{align*}
\overline{{\Heis}}_{S,n}
=\prod_{j\geq 0}{\Heis}_{n+j}\otimes\mathcal{O}_{S, -j},
\end{align*}
where ${\Heis}_{n+j}=0$ for $n+j<0$. It is also considered as 
a $\mathcal{D}_S$-module in an obvious way.
We also set 
\begin{align*}
{\End}({{\Heis}})_{S,m}\coloneqq \prod_{k\geq 0}\End({\Heis})_{m+k}\otimes \mathcal{O}_{S, -k}
\end{align*}
for $m\in \Z$ and ${\End}({{\Heis}})_{S}\coloneqq \bigoplus_m{\End}({{\Heis}})_{S,m}$.
\begin{lemma}\label{end H}
Every element in ${\End}({{\Heis}})_{S}$ 
defines an endomorphism on $\overline{{\Heis}}_S$.
\end{lemma}
\begin{proof}
Fix $n, m\in \Z$. We have an natural morphism from
$\overline{{\Heis}}_{S, n}\otimes \End({\Heis})_{S, m}$ to $\overline{{\Heis}}_{S,n+m}$
as follows:
\begin{align*}
\overline{{\Heis}}_{S, n}\otimes \End({\Heis})_{S, m}
&\simeq \prod_{\ell=0}^\infty\bigoplus_{\substack{j, k\geq 0\\ j+k=\ell}
}({\Heis}_{n+j}\otimes \End({\Heis})_{m+k})\otimes (\mathcal{O}_{S,-j}\otimes\mathcal{O}_{S,-k})\\
&\to \prod_{\ell=0}^\infty{\Heis}_{n+m+\ell}\otimes \mathcal{O}_{S,-\ell}
=\overline{{\Heis}}_{n+m}
\end{align*}
This gives the conclusion.
\end{proof}

Set 
\begin{align*}
\varphi^{(r)}_\lambda\coloneqq \frac{1}{2\kappa}\sum_{j=1}^r\frac{\lambda_ja_{-j}}{j}\in 
\End({{\Heis}})_{S, 0}.
\end{align*}

\begin{corollary}
The exponential 
\begin{align*}
\Phi^{(r)}_\lambda\coloneqq \exp\(\varphi^{(r)}_\lambda\)=\sum_{n=0}^\infty\frac{1}{n!}\(\varphi^{(r)}_\lambda\)^n
\end{align*}
defines an automorphism on $\overline{{\Heis}}_{S}$.
\end{corollary}

\begin{definition}
\label{def:Fock_lam}
We set 
${\Heis}^{(r)}
\coloneqq \Phi_\lambda^{(r)}({\Heis}\otimes \mathcal{O}_S)
\subset \overline{{\Heis}}_S$, 
and $\Coh\coloneqq \Phi_\lambda^{(r)}(\vac\otimes 1)$.
\end{definition}
We shall show that the pair $({\Heis}^{(r)}, \Coh)$
is equipped with the structure of non-singular coherent state ${\Heis}_\kappa$-module over $S$.

\begin{lemma}\label{Heis. diff.}
We have $\left[2\kappa n\partial_{\lambda_n},\Phi_\lambda^{(r)}\right]=a_{-n}\Phi_{\lambda}^{(r)}$
for $n=1,\dots, r$.
\end{lemma}
By this lemma, we obtain the following.
\begin{corollary}
${\Heis}^{(r)}$ is a $\mathcal{D}_S$-submodule of $\overline{{\Heis}}_S$. 
\end{corollary}

\begin{lemma}\label{Heis. prod}
As endomorphisms on $\overline{{\Heis}}_S$, we have the commutation relation
\begin{align*}
\left[a_{n},\Phi_\lambda^{(r)}\right]=
\begin{cases}
\lambda_n&(0<n\leq r)\\
0&(\text{otherwise})
\end{cases}.
\end{align*}
\end{lemma}
\begin{proof}
By definition, we have
\begin{align*}
\left[\varphi_\lambda^{(r)} ,a_{n}\right]=-\sum_{j=1}^r\lambda_j\delta_{j,n}\id_{\overline{{\Heis}}_S}.
\end{align*}
Hence, we obtain 
\begin{align*}
\Phi_\lambda^{(r)}a_{n}\(\Phi_\lambda^{(r)}\)^{-1}
&=\sum_{k=0}^\infty\frac{1}{k!}\(\mathrm{ad}_{\varphi_\lambda^{(r)}}\)^ka_{n}\\
&=a_{n}-\sum_{j=1}^r\lambda_j\delta_{j,n}\id_{\overline{{\Heis}}_S}
\end{align*}
This implies the lemma.
\end{proof}

\begin{corollary}\label{basis H}
We have 
\begin{align*}
{\Heis}^{(r)}=\bigoplus_{n_1\geq n_2\geq \cdots\geq n_k>0}\mathcal{O}_S\,a_{-n_1}\cdots a_{-n_k}\Coh
\end{align*}
as an $\mathcal{O}_S$-module.
\end{corollary}
It also follows from Lemma \ref{Heis. prod} that $a_{n}$ $(n\in\Z)$ acts on ${\Heis}^{(r)}$.
We can define 
\begin{align*}
Y_{{\Heis}^{(r)}}(\cdot, z) \colon 
{\Heis}\longrightarrow \End_{\mathcal{D}_S}({\Heis}^{(r)})\jump{z^{\pm 1}}
\end{align*}
in a way similar to (\ref{field Heis.}).

\begin{proposition}\label{csm H}
The tuple
${\Heis}^{(r)}_\kappa\coloneqq \({\Heis}^{(r)}, Y_{{\Heis}^{(r)}}(\cdot, z), \Coh\)$ 
is a non-singular coherent state 
${\Heis}$-module, which is an envelope of ${\Heis}_\kappa$.
\end{proposition}
\subsection{Irregular vertex algebra structure}
\begin{lemma}\label{filter-H}
Let $F^\bullet({\Heis}^{(r)})$ be the decreasing filtration on ${\Heis}^{(r)}$ 
defined by
\begin{align*}
F^k({\Heis}^{(r)}_n)\coloneqq \Phi_\lambda^{(r)}\(\bigoplus_{j\geq k}{\Heis}_{n+j}\otimes \mathcal{O}_{S, -j}\).
\end{align*}
Then, $({\Heis}^{(r)}, F^\bullet) $ is a filtered small lattice of ${\Heis}^{(r)}$ i.e. 
satisfies the conditions in Definition \ref{FSL}. 
\end{lemma}
\begin{proof}
The condition (L) requires nothing since $H$ is empty.
Since $F^0({\Heis}^{(r)})$ equals to ${\Heis}^{(r)}$, condition (F1) holds.
The conditions (F2), (F4) (resp. (F3), (F5)) are the corollaries of Lemma \ref{Heis. prod} (resp. Lemma \ref{Heis. diff.}). 
\end{proof}
\begin{lemma}\label{end-series}
We have an isomorphism
\begin{align*}
\End({\Heis})_S\jump{z^{\pm 1}}_n\simeq 
\prod_{k\geq0}\End({\Heis})\jump{z^{\pm}}_{k+n}\otimes \mathcal{O}_{S,-k}.
\end{align*}
\end{lemma}
\begin{proof}
We have 
\begin{align*}
\End({\Heis})_S\jump{z^{\pm 1}}_n
&\simeq 
\prod_{\ell\in\Z}\End({\Heis})_{S,n+\ell}z^{\ell}\\
&\simeq 
\prod_{\ell\in\Z}\prod_{k\geq 0}\End({\Heis}_{n+k+\ell})\otimes \mathcal{O}_{S,-k}z^\ell\\
&\simeq
\prod_{k\geq 0}\(\prod_{\ell\in\Z}\End({\Heis}_{n+k+\ell})z^\ell\)\otimes \mathcal{O}_{S,-k}
\end{align*}
This proves the lemma.
\end{proof}

Define $\overline{{\Heis}}_{S^2}$ and $\End({\Heis})_{S^2}$ by replacing $S$ with $S^2$ in the definition of 
$\overline{{\Heis}}_S$ and $\End({\Heis})_S$, respectively. 
We can replace $S$ with $S^2$ in Lemma \ref{end H} and Lemma \ref{end-series}. 
We shall consider the following extension of $Y(\cdot, z)$:
\begin{definition}
Define the morphism of $\mathcal{D}_{S^2}$-modules
\begin{align}
\overline{Y}(\cdot, z) \colon \overline{{\Heis}}_{S^2}\longrightarrow 
\End({\Heis})_{S^2}\jump{z^{\pm 1}}
\end{align}
by 
$\overline{Y}\(\sum_{k=0}^\infty A_{n+k}\otimes P_{-k}(\lambda,\mu), z\)
\coloneqq \sum_{k=0}^\infty Y(A_{k+n}, z)\otimes P_{-k}(\lambda,\mu)$
for each homogeneous element 
$\sum_{k=0}^\infty A_{n+k}\otimes P_{-k}(\lambda,\mu)\in {\Heis}_{S^2, n}$.
\end{definition}
We identify $\overline{{\Heis}}_{S_\lam}$ 
with a subspace of $\overline{{\Heis}}_{S^2_{\lam,\mu}}$ along with 
the $S_\lam$-axis. 
Then $\mathcal{D}_{S_\lam}$-module structures on $\overline{{\Heis}}_{S_\lam}$ and 
$\overline{{\Heis}}_{S^2_{\lam,\mu}}$ are compatible. 
We also regards the coherent state module $\Heis_\lam^{(r)} \subset \overline{{\Heis}}_{S_\lam}$ 
defined in Definition \ref{def:Fock_lam} as a subspace of $\overline{{\Heis}}_{S^2_{\lam,\mu}}$ 
through the above identification. 
We do the same thing above for $\mu$. 

We note that the pull back
${\Heis}^{(r)}_{\lambda+\mu}\coloneqq \sigma_{\lambda+\mu}^*{\Heis}^{(r)}$
(see Section \ref{Notation})
is a $\mathcal{D}_{S^2_{\lam,\mu}}$-submodule of $\overline{{\Heis}}_{S^2_{\lam,\mu}}$ by definition.
The completion $\overline{{\Heis}}^{(r)}_{\lambda+\mu}$ 
of ${\Heis}^{(r)}_{\lambda+\mu}$ is naturally identified with $\overline{{\Heis}}_{S^2_{\lam,\mu}}$. 
By Lemma \ref{end H}, we can restrict $\overline{Y}(\cdot, z)$ to ${\Heis}^{(r)}_\lambda$:
\begin{align}\label{I-H}
\ssc(\cdot, z) \colon {\Heis}^{(r)}_\lambda\longrightarrow 
\Hom_{\mathcal{D}_{S_\mu}}\({\Heis}^{(r)}_{\mu},\overline{{\Heis}}^{(r)}_{\lambda+\mu}\)\jump{z^{\pm 1}}
\end{align}
and $\ssc(\cdot, z)$ is a $\mathcal{D}_{S_\lam}$-module morphism by definition.

Set 
\begin{align*}
\varphi_\lambda^{(r)}(z)_{\pm}=
\frac{1}{2\kappa}\sum_{n=1}^r \frac{\lambda_n\partial_z^{(n-1)}a(z)_\pm}{n}.
\end{align*}
\begin{lemma}\label{field formula}
There are equalities
\begin{align*}
\ssc(\coh{\lambda}, z)&=\exp\(\varphi_\lambda^{(r)}(z)_+\)\exp\(\varphi_\lambda^{(r)}(z)_-\),
\text{ and}\\
\ssc\(a_{-n_1}\cdots a_{-n_k}\coh{\lambda},z\)
&=\nop{Y_{{\Heis}^{(r)}}(a_{-n_1}, z)\cdots Y_{{\Heis}^{(r)}}(a_{-n_k},z)
\ssc(\coh{\lambda}, z)}
\end{align*}
for $k, n_1,\dots, n_k\in \Z_{>0}$. 
\end{lemma}
\begin{proof}
By definition, we have
\begin{align*}
\coh{\lambda}
=&\sum_{j=0}^\infty \frac{1}{j!}\(\frac{1}{2\kappa}\sum_{k=1}^r\frac{\lambda_ka_{-k}}{k}\)^j\vac\\
=&\sum_{(j_k)_{k=1}^r\in\Z_{\geq 0}^{\oplus r}}\prod_{k=1}^r\(\frac{1}{(2\kappa)^{j_k}}\frac{\lambda_k^{j_k} a_{-k}^{j_k}}{j_k!k^{j_k}}\)\vac
\end{align*}
and hence
\begin{align*}
\ssc(\coh{\lambda}, z)&=
\sum_{(j_k)_{k=1}^r\in\Z_{\geq 0}^{\oplus r}}
\prod_{k=1}^r
\(\frac{1}{(2\kappa)^{j_k}}\frac{\lambda_k^{j_k}}{j_k!}\)
Y\(\prod_{k=1}^r\frac{a_{-k}^{j_k}}{k^{j_k}}\vac, z\)\\
&=\sum_{(j_k)_{k=1}^r\in\Z_{\geq 0}^{\oplus r}}\(\prod_{k=1}^r\frac{\lambda_k^{j_k}}{j_k!}\)
\nop{\prod_{k=1}^r\(\frac{1}{\kappa}Y\(\frac{a_{-k}}{k},z\)\)^{j_k}}\\
&=\exp\(\varphi_\lambda^{(r)}(z)_+\)\exp\(\varphi_\lambda^{(r)}(z)_-\).
\end{align*}
This proves the first equality.
The second equality can also be proved similarly.
\end{proof}
Note that we have
 $\varphi_\lambda^{(r)}(z)_+|_{z=0}=\varphi_\lambda^{(r)}$ and $\varphi_\lambda^{(r)}(z)_+\vac=0$.
\begin{corollary}\label{vac. Heis.}
For any $\mathcal{A}_\lambda\in{\Heis}^{(r)}_\lam$, we have
$\ssc(\mathcal{A}_\lambda, z)\vac\in {\Heis}^{(r)}_\lam\jump{z}$, 
and \[\ssc(\mathcal{A}_\lambda, z)\vac|_{z=0}=\mathcal{A}_\lambda.\]
\end{corollary}

\begin{lemma}\label{field H}
Put \begin{align*}
\irr_\kappa(z;\lambda,\mu)\coloneqq 
\frac{1}{2\kappa}\sum_{1\leq p, q\leq r}\binom{p+q}{p}\frac{(-1)^{p+1}}{p+q}\frac{\lambda_p\mu_q}{z^{p+q}}.
\end{align*}
We have $e^{-\irr_\kappa(z;\lambda,\mu)}\ssc(\coh{\lambda}, z)\coh{\mu}\in {\Heis}^{(r)}_{\lambda+\mu}\jump{z}$, 
and 
$e^{-\irr_\kappa(z;\lambda,\mu)}\ssc(\coh{\lambda}, z)\coh{\mu}|_{z=0}=\coh{\lambda+\mu}$.

Moreover, for any $\mathcal{A}_\lambda \in {\Heis}_\lambda^{(r)}$, 
$\ssc(\mathcal{A}_\lambda, z)$ is an 
irregular field on ${\Heis}^{(r)}_\mu$ with the irregularity $\irr_\kappa(z;\lambda, \mu)$.
\end{lemma}
\begin{proof}
Since we have $[\partial_z^{(p-1)}a(z)_+, a_{-q}]=0$ and
\begin{align}\label{comm z0}
\left[\partial_z^{(p-1)}a(z)_-, a_{-q}\right]=2\kappa\frac{(p+q-1)!}{(p-1)!(q-1)!}\frac{(-1)^{p-1}}{z^{p+q}}\id
\end{align}
for $p, q\in \Z_{>0}$, 
we have $[\varphi_\lambda^{(r)}(z)_+, \varphi_\mu^{(r)}]=0$ and
\begin{align*}
\left[\varphi_\lambda^{(r)}(z)_-, \varphi_\mu^{(r)}\right]=
\irr_\kappa(z;\lambda, \mu)\cdot\id_{\overline{{\Heis}}_{S^2}}.
\end{align*}
Using this equality, we obtain that  
\begin{align*}
\ssc(\coh{\lambda}, z)\coh{\mu}
&=\exp\(\varphi_\lambda^{(r)}(z)_+\)
\exp\(\varphi_\lambda^{(r)}(z)_- \) e^{\varphi_\mu^{(r)}}\vac\\
&=e^{\irr_\kappa(z;\lambda,\mu)}e^{\varphi_\mu^{(r)}}\exp\(\varphi_\lambda^{(r)}(z)_+\) \vac
\end{align*}
We obtain the first statement by Corollary \ref{vac. Heis.}.

The latter statement can be proved by using 
Corollary \ref{basis H}, Lemma \ref{field formula} and (\ref{comm z0}).
\end{proof}

\begin{lemma}\label{coherent local H}
There is an equality
\begin{align*}
&e^{-\irr_\kappa(z-w;\lambda,\mu)}_{|z|>|w|}\ssc(\coh{\lambda}, z)\ssc(\coh{\mu}, w)\\
=&\exp\(\varphi_\lambda^{(r)}(z)_+ + \varphi_\mu^{(r)}(w)_+\)
\exp\(\varphi_\lambda^{(r)}(z)_-+ \varphi_\mu^{(r)}(w)_-\).
\end{align*}
\end{lemma}
\begin{proof}
We have $ [\partial_z^{(p-1)}a(z)_+,\partial_w^{(q-1)}a(w)_+]=[\partial_z^{p-1}a(z)_-,\partial_w^{(q-1)}a(w)_-]=0$, 
\begin{align*}
[\partial_z^{(p-1)}a(z)_-,\partial_w^{(q-1)}a(w)_+]=2\kappa\frac{(p+q-1)!}{(p-1)!(q-1)!}\frac{(-1)^{p-1}}{(z-w)^{p+q}}|_{|z|>|w|}\id
\end{align*}
where $(z-w)^{-p-q}|_{|z|>|w|}$ denotes the expansion in positive powers of $w/z$.
Hence we obtain 
\begin{align*}
\left[\varphi_\lambda^{(r)}(z)_-, \varphi_\mu^{(r)}(w)_+ \right]=\irr_\kappa(z-w;\lambda,\mu)|_{|z|>|w|}.
\end{align*}
By the Baker-Campbell-Hausdorff formula, we obtain the lemma.
\end{proof}

Since  $[T,\Phi_\lambda^{(r)}]=\sum_{j=1}^r \lambda_j a_{-j-1}\Phi_{\lambda}^{(r)}$, 
$T$ naturally acts on ${\Heis}^{(r)}$.
\begin{theorem}
The tuple $\left({\Heis}_\kappa^{(r)},({\Heis}^{(r)}, F^\bullet), \ssc(\cdot, z),T, \irr_\kappa\right)$
is an irregular vertex algebra for ${\Heis}_\kappa$.
\end{theorem}
\begin{proof}Let us check the axioms in Definition \ref{IVA}.
The translation axiom and the compatibility condition are trivial by the construction.
The vacuum axiom is proved in Corollary \ref{vac. Heis.}. 
The irregular field axiom and coherent state axiom are proved in Lemma \ref{field H}.
It remains to prove the irregular locality axiom.

By Lemma \ref{coherent local H}, $\ssc(\coh{\lambda}, z)$ and $\ssc(\coh{\mu}, w)$
are mutually $\irr_\kappa$-local. 
Then, by the compatibility, we can apply the Dong's lemma (Lemma \ref{Dong}), 
to obtain the $\irr_\kappa$-locality in general.
\end{proof}

\subsection{Conformal structures}
In this subsection, we shall show that ${\Heis}^{(r)}$
is an irregular vertex operator algebra if $\kappa=1/2$. 
Recall that the space $S=\Spec\C[\lambda_j]_{j=1}^r$
is equipped with the $\Der_0(\C\jump{t})$-structure
as explained in Example \ref{S}.
We firstly prove the following:
\begin{lemma}\label{irregularity conformality}
The irregularity
$\irr_\kappa(z; \lambda, \mu)$ is conformal (see Definition \ref{conformal irregularity}). 
\end{lemma}
\begin{proof}
We have
\begin{align*}
[D_j^{\mu}, \irr_\kappa(z;\lambda,\mu)]=
\frac{1}{2\kappa}\sum_{k,\ell>0}(-1)^{k-1}\binom{k+\ell-1}{\ell-1}\frac{\lambda_{k}\mu_{\ell+j}}{z^{k+\ell}}.
\end{align*}
Since $\partial_z^{(m+1)}z^{j+1}=\binom{j+1}{m+1}z^{j-m}$
 for $ m\geq -1$, we have
\begin{align*}\begin{split}
&\partial_z^{(m+1)}z^{j+1}[D_m^\lambda,\irr_\kappa(z;\lambda,\mu)]\\
=&\frac{1}{2\kappa}\sum_{p, q>0}(-1)^{p-1}\binom{p+q-1}{p-1}\binom{j+1}{m+1}\frac{\lambda_{p+m}\mu_{q}}{z^{p+q-j+m}}.
\end{split}
\end{align*}

Hence the coefficient of $\lambda_u\mu_v/z^{w}$ ($u, v, w>0$) with $u+v=j+w$ in (\ref{m seq}) 
is given by
\begin{align*}
(-1)^{u-1}\binom{w-1}{v-j-1}
+
\sum_{\substack{ s=0}}^u(-1)^s\binom{v+s}{s}\binom{j+1}{u-s}.
\end{align*}
Hence we need to show 
\begin{align*}
(-1)^{u}\binom{w-1}{u}=\sum_{\substack{ s=0}}^u(-1)^s\binom{v+s}{s}\binom{j+1}{u-s}.
\end{align*}
The left hand side of this equation is the coefficients of $x^u$ in $(1+x)^{-(w-u)}$, 
and the right hand side 
is that of $x^u$ in $(1+x)^{-(v+1)}(1+x)^{j+1}=(1+x)^{j-v}$.
Hence we obtain the lemma.
\end{proof}

Let the complex number $\kappa$ be $1/2$. 
Take a complex number $\rho$, 
and put $c=1-12\rho^2$.
It is known that $\omega_\rho\coloneqq \frac{1}{2}a_{-1}^2+\rho a_{-2}$ 
is a conformal vector of the Heisenberg vertex algebra ${\Heis}$.
Let $h_k$ $(k=0,\dots, 2r)$ (and hence $\mathcal{L}_k$)
as in Section \ref{Vir csm} with $\lambda_0=0$. 
For the simplicity of the notation, we denote 
\begin{align*}
Y_{{\Heis}^{(r)}}(\omega_\rho, z)=\sum_{n\in\Z}L_n z^{-n-2}.
\end{align*}

\begin{lemma}\label{L-bracket}
For $s\geq 0$, we have
\begin{align*}
\left[L_s-\mathcal{L}_s,\Phi_\lambda^{(r)}\right]=h_s\Phi_\lambda^{(r)}
+\sum_{k=0}^s \lambda_k\Phi_\lambda^{(r)} a_{s-k}.
\end{align*}
\end{lemma}
\begin{proof}
By Lemma \ref{Heis. prod}, we have
\begin{align*}\begin{split}
\left[L_s,\Phi_\lambda^{(r)}\right]
&=
\left[\frac{1}{2}\sum_{k=0}^sa_{k}a_{s-k}+\sum_{p>0}a_{-p}a_{s+p}-\rho (s+1)a_{s},
\Phi_\lambda^{(r)}\right]\\
&=\frac{1}{2}\sum_{k=0}^s[a_{k}a_{s-k},\Phi_\lambda^{(r)}]
+\sum_{p>0}a_{-p}[a_{s+p},\Phi_\lambda^{(r)}]
-\rho(s+1)[a_{s},\Phi_\lambda^{(r)}]\\
&=h_s\Phi_\lambda^{(r)}
+\sum_{k=0}^s \lambda_k\Phi_\lambda^{(r)} a_{s-k}
+
\sum_{p>0}a_{-p}\lambda_{s+p}\Phi_\lambda^{(r)}
\end{split}
\end{align*}
On the other hand, by Lemma \ref{Heis. diff.}
we have 
\begin{align*}
\left[\mathcal{L}_s,\Phi_\lambda^{(r)}\right]
&=\left[\sum_{p>0}p\lambda_{s+p}\partial_{\lambda_p},\Phi_\lambda^{(r)}\right]\\
&=\sum_{p>0}\lambda_{s+p} a_{-p}\Phi_\lambda^{(r)}.
\end{align*}
This proves the lemma.
\end{proof}
\begin{corollary}
For $k\in\Z_{\geq 0}$, we have
\begin{align}\label{CONFORMAL}
&L_k\coh{\lambda}=\mathcal{L}_k\coh{\lambda}\\
&L_{-k}\coh{\lambda}=\sum_{j=1}^r\lambda_ja_{-j-k}\coh{\lambda}+\rho (k+1)a_{-k}\coh{\lambda}
+\frac{1}{2}\sum_{\ell=1}^k a_{-\ell}a_{-k+\ell}\coh{\lambda}.
\end{align}
\end{corollary}

\begin{proposition}
The irregular Heisenberg vertex algebra ${\Heis}^{(r)}$
for $\kappa=1/2$ is an irregular vertex operator algebra.
\end{proposition}
\begin{proof}
Lemma \ref{gyaku virasoro} shows that 
$\rho_{{\Heis}^{(r)}};t^{k+1}\partial_t\mapsto -(L_k-\mathcal{L}_k)$ is a Lie algebra homomorphism.

For $v_{k+d}\in {\Heis}_{k+d}$ and $f_k(\lambda)\in \mathcal{O}_{S,-k}$, 
by Lemma \ref{L-bracket}, we have
\begin{align*}
&(L_0-\mathcal{L}_0)(\Phi_\lambda^{(r)} (v_{k+d}\otimes f_k(\lambda)))\\
&=\Phi_\lambda^{(r)} (L_0v_{k+d}\otimes f_k(\lambda))-\Phi_\lambda^{(r)} (v_{k+d}\otimes \mathcal{L}_0 f_k(\lambda))\\
&=(k+d)\Phi_\lambda^{(r)}(v_{k+d}\otimes f_k(\lambda))-k\Phi_\lambda^{(r)}(v_{k+d}\otimes f_k(\lambda))\\
&=d\Phi_\lambda^{(r)}(v_{k+d}\otimes f_k(\lambda))
\end{align*}
This proves that
$L_0-\mathcal{L}_0$ acts as the grading operator on ${\Heis}^{(r)}$.  

By Lemma \ref{L-bracket},  for any $v_\lambda\in {\Heis}\otimes\mathcal{O}_S$, we have
\begin{align*}\label{com phi}
\begin{split}
(L_s-\mathcal{L}_s) \Phi_\lambda^{(r)}(v_\lambda)
&=
[L_s-\mathcal{L}_s, \Phi_\lambda^{(r)}]v_\lambda+\Phi_\lambda^{(r)}(L_s(v_\lambda)-\mathcal{L}_s(v_\lambda))\\
&=\Phi_{\lambda}^{(r)}\(
\sum_{k=1}^{s-1}\lambda_{s-k}a_{k}+
L_s-D_s\)v_{\lambda}.
\end{split}
\end{align*}
Since 
\[\sum_{k=1}^{s-1}\lambda_{s-k}a_{k}+
L_s-D_s\]
 is locally nilpotent on $\Heis\otimes \O_S$, 
we can deduce that $L_s-\mathcal{L}_s$ is locally nilpotent.
\end{proof}

\section{Irregular Virasoro vertex operator algebras}\label{ffr}
We shall give a definition of irregular Virasoro vertex algebra via the free field realization.

\subsection{Saturated vertex subalgebras}
Let 
$( \mathcal{V}, \ssc,\irr)$ be a non-singular irregular vertex operator algebra 
for a vertex operator algebra $V$ on $S$.
In particular, a filtration $F^\bullet(\mathcal{V})$ with the properties (F1)-(F4) is fixed.
Let $W\subset V$ be a vertex operator subalgebra of $V$. 
Let $\env\subset \mathcal{V}$ be the smallest $\mathcal{D}_S$-submodule which contains $\Coh\in \mathcal{V}$ 
and is closed under the operation $A_{(n)}^{\mathcal{V}}$
for every $A\in W$ and every $n\in \Z$.

\begin{definition}
The vertex operator subalgebra $W\subset V$ is 
\textit{saturated with respect to $\mathcal{V}$} if
$\env$ is a coherent state $W$-module with singularity $H$ 
and $\env(*H)=\mathcal{V}(*H)$. 
\end{definition}
If $W$ is saturated, 
set $\env^\circ\coloneqq \mathcal{V}\subset \env(*H)$ and $F^\bullet(\env^\circ)\coloneqq F^\bullet(\mathcal{V})$.
Then we have the following lemma:
\begin{lemma}\label{saturated lemma}
$\big(\env, (\env^\circ, F^\bullet), \ssc, \irr\big)$ is 
an irregular vertex operator algebra for $W$.
\end{lemma}
\begin{proof}
The only non-trivial point is that 
$(\env, Y_\env, (\env^\circ, F^\bullet))$ is an envelope of $W$. 
In other words, we need to show that
the morphism
\[\overline{\Psi}_W\colon W\longrightarrow \env_\O|^\circ_0
=\env_\O/\(\env_\O\cap\mathfrak{m}_{S,0}\mathcal{V}\),\]
defined in Definition \ref{def:envelope}, is an isomorphism. 
Since $\mathcal{V}$ is non-singular, 
we have an isomorphism
\[\overline{\Psi}_V\colon V\longrightarrow \mathcal{V}/\mathfrak{m}_{S,0}\mathcal{V}. \]
Since $\overline{\Psi}_W$ is the restriction of 
$\overline{\Psi}_V$, it is injective. 

It remains to prove that $\overline{\Psi}_W$ is surjective. 
Since $\overline{\Psi}_V$ is an isomorphism of $V$-modules,
we have
\begin{align}\label{N>-1}
A_{(n)}^{\mathcal{V}}\Coh
\in\mathfrak{m}_{S,0}\mathcal{V}
\end{align}
for $A\in V$ and $n\geq 0$.
By the construction, a section of $\env_\O$ 
can be expressed as an $\O_S$-linear combination 
of 
the sections of the form
\begin{align}\label{ACOH}
A_{1,(n_1)}\cdots A_{k,(n_k)}\Coh
\end{align}
for some $A_1,\dots, A_k\in W$, and $n_1\leq \cdots \leq n_k\in \Z$. 
If $n_k\leq -1$, then 
(\ref{ACOH})
is the image of 
\[A_{1,(n_1)}\cdots A_{k,(n_k)}\vac\]
by ${\Psi}_W\colon W\to \env_\O,\quad A\mapsto A_{(-1)}\Coh$.
If $n_k\geq 0$, then by (\ref{N>-1}), 
the class of (\ref{ACOH}) in $\env_\O|^\circ_0$ is zero.
Hence we obtain the lemma.
\end{proof}

\subsection{Irregular Virasoro vertex algebra via free field realization}\label{I Virasoro}
Recall that $\Vir_c$ denotes the Virasoro vertex algebra (Section \ref{ccsm}).
Let ${\Heis}$ be the Heisenberg vertex algebra with $\kappa\coloneqq 1/2$. 
The irregular Heisenberg algebra ${\Heis}^{(r)}$
is also considered in the case $\kappa=1/2$.

Consider the 
morphism 
$\Vir_c\to {\Heis}$ given by the 
conformal vector $\omega=\frac{1}{2}a_{-1}^2+\rho a_{-2}$
for a complex number $\rho$ and $c=1-12\rho^2$.
We assume that $c$ is generic so that 
we have $\Vir_c\subset \Heis$. 
\begin{definition}
Let $\Vir_c^{(r)}$ denote the 
smallest $\mathcal{D}_S$-submodule of ${\Heis}^{(r)}$
which is closed under all operations of 
the form $A_{(n)}^{{\Heis}^{(r)}}$ for $A\in \Vir_c$ and $n\in \Z$.
\end{definition}

Let $\csm_{c,0}^{(r)}$ denote the coherent state module defined in 
Section \ref{Vir csm} with $\lambda_0=0$. 
The relation between $\Vir_c^{(r)}$ and  $\csm_{c,0}^{(r)}$  is given by 
the following proposition: 
\begin{proposition}
We have a unique morphism 
\begin{align}\label{VIR}
\mathcal{M}^{(r)}_{c, 0}\longrightarrow {\Heis}^{(r)}
\end{align}
of $\mathcal{D}_S\otimes_\C U(\Vir)$-modules such that 
the coherent state of $\mathcal{M}^{(r)}_{c, 0}$ maps to that of $ {\Heis}^{(r)}$.
The morphism $(\ref{VIR})$ is injective and the image of $(\ref{VIR})$ coincides with $\Vir_c^{(r)}$.
\end{proposition}
\begin{proof}
By Corollary \ref{CONFORMAL}, 
there is a unique morphism
$\mathcal{D}_S\otimes_\C U(\Vir)\longrightarrow {\Heis}^{(r)}$
which sends $1\otimes 1$ to $\coh{\lam}$.
By (\ref{CONFORMAL}),
the above morphism 
uniquely induces 
the morphism (\ref{VIR}). 
On the one hand, by Remark \ref{simple}, 
the morphism (\ref{VIR}) is injective at each point $\lam^o=(\lam_1^o,\dots,\lam_r^o)$ with
$\lam_r^o\neq 0$. 
On the other hand, since $\mathcal{M}^{(r)}_{c, 0}$ is a free $\O_S$-module (see (\ref{qcM})), 
the kernel of (\ref{VIR}) should be torsion free. 
Hence we have 
that the kernel is the zero module, which means that (\ref{VIR}) is injective. 
The coincidence of the image with $\Vir_c^{(r)}$
follows from the minimality in the definition of $\Vir_c^{(r)}$.
\end{proof}

\begin{corollary}
$\Vir_c^{(r)}$ is a coherent state $\Vir_c$-module with singularity $\{\lambda_r=0\}$.
\end{corollary}

The following theorem is 
the main theorem of this section.

\begin{theorem}\label{VIRASORO}
The Virasoro vertex algebra 
$\Vir_c$ is saturated in the Heisenberg vertex algebra
${\Heis}$
with respect to the irregular vertex algebra
${\Heis}^{(r)}$.
\end{theorem}
\begin{proof} We have already checked that 
$\Vir_c^{(r)}$ is coherent state $\Vir_c$-module with singular divisor $H=\{\lambda_r=0\}$.
Hence it remains to show that $\Vir_c^{(r)}(*H)={\Heis}^{(r)}(*H)$.
We shall prove this by showing the following proposition inductively on $n$:
\begin{itemize}
\item[$(\bm{P}_n)$]
Every section $a_{-n_1}\cdots a_{-n_k}\coh{\lambda}$ with $n_j>0$ $(1\leq j\leq k)$ and  
$\sum_{j=1}^kn_k\leq n$ is in $\Vir_c^{(r)}(*H)$.
\end{itemize}
The proposition $(\bm{P}_0)$ is trivial since $\coh{\lambda}\in \Vir_c^{(r)}$.
Assume that $(\bm{P}_{n-1})$ holds. 
Let $a_{-n_1}\cdots a_{-n_k}\coh{\lambda}$ be
an arbitrary section with $\sum_{j=1}^kn_j=n$. 
If $0<n_1\leq r$, then 
\begin{align*}
a_{-n_1}\cdots a_{-n_k}\coh{\lambda}
= n_1\frac{\partial}{\partial{\lambda_{n_1}}}a_{-n_2}\cdots a_{-n_k}\coh{\lambda}
\in \Vir_c^{(r)}(*H).
\end{align*}
If $n_1>r$, consider the action of $L_{-n_1+r}$ on $a_{-n_2}\cdots a_{-n_k}\coh{\lambda}$: 
\begin{align*}
&L_{-n_1+r}a_{-n_2}\cdots a_{-n_k}\coh{\lambda}\\
=&\sum_{j=2}^ka_{-n_2}\cdots a_{-n_{j-1}}[L_{-n_1+r}, a_{-n_j}]a_{-n_{j+1}}\cdots a_{-n_k}\coh{\lambda}\\
&+a_{-n_2}\cdots a_{-n_k}L_{-n_1+r}\coh{\lambda}\\
=&\sum_{j=2}^ka_{-n_2}\cdots a_{-n_{j-1}}n_j a_{-n_j-n_1+r}a_{-n_{j+1}}\cdots a_{-n_k}\coh{\lambda}\\
&+\(\sum_{\ell=1}^{r}\lambda_\ell a_{-n_1+r-\ell}
+\rho(n_1-r-1)a_{-n_1+r}\)a_{-n_2}\cdots a_{-n_k}\coh{\lambda}\\
&+\frac{1}{2}\sum_{i=1}^{n_1-r}a_{-n_2}\cdots a_{-n_k}a_{-i}a_{-n_1+r+i}\coh{\lambda}.
\end{align*}
Each term other than $\lambda_r a_{-n_1}a_{-n_2}\cdots a_{-n_k}\coh{\lambda}$
is in $\Vir_c^{(r)}(*H)$ by $(\bm{P}_{n-1})$. 
We also have $L_{-n_1+r}a_{-n_2}\cdots a_{-n_k}\coh{\lambda}\in \Vir_c^{(r)}(*H)$. 
Hence we obtain that the element $ a_{-n_1}\cdots a_{-n_k}\coh{\lambda}$ is in $\Vir_c^{(r)}(*H)$. 
This proves $(\bm{P}_n)$ and hence the theorem.
\end{proof}
By the Lemma \ref{saturated lemma}, 
we can define the irregular Virasoro vertex algebra: 
\begin{definition}
We set $\mathcal{V}ir^{(r)}_c\coloneqq {\Heis}^{(r)}$, which is considered as a filtered small lattice of 
$\Vir_c^{(r)}(*H)$.
The irregular vertex operator algebra 
$$\Vir_c^{(r)}\coloneqq \(\Vir_c^{(r)}, \(\mathcal{V}ir_c^{(r)}, F^\bullet\), \ssc, \irr\)$$
is called an \textit{irregular Virasoro vertex algebra}. 
\end{definition}

\begin{remark}
The quotient 
\[M_{c,h}=(\csm_{c,h}^{(r)})_\O/\((\csm_{c, h}^{(r)})_\O\cap\mathfrak{m}_{S,0}\csm_{c, h}^{(r)}\)\]
is isomorphic to the usual Verma module for the Virasoro algebra,
while the quotient 
\[(\Vir_c^{(r)})_\O|^\circ_0= (\Vir_c^{(r)})_\O/\((\Vir_c^{(r)})_\O\cap\mathfrak{m}_{S, 0}\Heis^{(r)}\)\]
is isomorphic to $\Vir_c$ via $\overline{\Psi}_{\Vir_c}$. 
\end{remark}
\begin{remark}
 By the Theorem \ref{VIRASORO}, at least theoretically, we can describe the vertex operators $Y(\mathcal{A}_\lam, z)$
for $\mathcal{A}_\lam\in \mathcal{V}ir_c^{(r)}$ only in terms of 
the Virasoro algebra and the coherent states after the localization,
although the computation is very complicated in practice.
For example, 
in the case $r=1$ and $\rho=0$, we have
\begin{align*}
Y(\coh{\lam},z)\coh{\mu}
=&e^{\lam\mu/z^2}\(1+\lam a_{-2}z+
\(\frac{\lam^2 a_{-2}^2}{2!}+\lam a_{-3}\)z^2+\cdots\)\coh{\lam+\mu}\\
=&e^{\lam\mu/z^2}\(1+\frac{\lambda}{\lam+\mu} L_{-1}z\right.+\\
&+\left.\frac{1}{(\lam+\mu)^2}
\(\frac{\lam^2L_{-1}^2}{2!}+\lam\mu\(L_{-2}-\frac{L_0^2-L_0}{2(\lam+\mu)^2}\)\)z^2+\cdots\)\coh{\lam+\mu},
\end{align*}
where we put $\lam=\lam_1$ and $\mu=\mu_1$.
\end{remark}

\section*{Concluding remarks}
The key point of our construction in Section \ref{FFIVA}
was the Baker-Campbell-Hausdorff formula used in
Lemma \ref{coherent local H}.
We expect that our way of constructing an irregular vertex algebra used in 
Section \ref{FFIVA} can be easily generalized to the vertex algebras 
generated by finitely many free fields in the sense of \cite[Definition]{Kac}
although we have only treated the case where the vertex algebra is generated by a free field, 
for simplicity 
(and to give a canonical conformal structure).

Then, we also expect that Lemma \ref{saturated lemma} in Section \ref{ffr}
(or its generalization) would be useful in the construction of 
irregular versions of Kac-Moody vertex algebras 
and $\mathcal{W}$-algebras.
In other words, we expect that we may define irregular versions of 
$V_k(\ge)$ and $\mathcal{W}(\ge)$ via their respective free field realizations.
We shall refer \cite{N1,N2} and \cite{GL,KMST} as studies on irregular conformal blocks for
$V_k(\mathfrak{sl}_2)$ and $\mathcal{W}_3$-algebra, respectively. 
The details of these expectations would be given in 
the subsequent studies.

\end{document}